\documentclass[11pt]{article}

\usepackage{amsmath,amssymb,amsthm,mathrsfs}
\usepackage{hyperref}
\usepackage{fullpage}
\usepackage{xcolor}

\usepackage{tikz}
\usetikzlibrary{arrows.meta, positioning, shapes.geometric}

\newif\ifshowcomments
\showcommentsfalse       

\newtheorem{theorem}{Theorem}

\newtheorem{lemma}{Lemma}[section]
\newtheorem{proposition}{Proposition}[section]
\newtheorem{corollary}{Corollary}[section]

\theoremstyle{remark}
\newtheorem{remark}{Remark}[section]
\numberwithin{equation}{section}

\newcommand{\ud}{\,\mathrm{d}}

\newcommand{\R}{\mathbb R}
\newcommand{\E}{\mathbb E}
\newcommand{\Law}{\mathcal L}
\newcommand{\TV}{\mathrm{TV}}

\newcommand{\Ent}{\mathrm H}

\newcommand{\op}{\mathrm{op}}

\newcommand{\kconv}{\star}  
\usepackage{xparse}

\NewDocumentCommand{\Kop}{m g}{%
  \IfNoValueTF{#2}{%
    K \kconv #1%
  }{%
    (K \kconv #1)(#2)%
  }%
}

\title{Quantitative Convergence for Sequential Interacting Diffusions via Incremental Relative Entropy}

\author{
   Zhenfu Wang$^{1}$ and Xianliang Zhao$^{1}$ \\[2mm]
  {\small $^{1}$Beijing International Center for Mathematical Research, Peking University,}\\
  {\small 5 Yiheyuan Road, Beijing 100871, China}\\
  {\small \texttt{zwang@bicmr.pku.edu.cn} \quad \texttt{xzhaomath@gmail.com}}
}

\date{}

\begin{document}
\maketitle

\begin{abstract}
We study a lower-triangular system of interacting diffusions in which particle \(i\) interacts only with its predecessors through the empirical measure \(\mu^{i-1}_t\). This gives a directed, non-exchangeable approximation of the same McKean--Vlasov diffusion as the classical exchangeable particle system. We introduce an incremental path-space relative entropy adapted to the causal structure,
\[
R_i(T)
=
\Ent\!\left(P^{1:i}_{[0,T]}\,\middle|\,P^{1:i-1}_{[0,T]}\otimes \bar P_{[0,T]}\right),
\]
and prove the sharp estimate \(R_i(T)\lesssim (i-1)^{-1}\).
Furthermore, we obtain convergence of the empirical measure to the McKean--Vlasov law at the canonical \(N^{-1/2}\) scale in negative Sobolev norms. The proof combines a Girsanov representation, a martingale-difference replacement of predecessor empirical measures by averaged conditional measures, an upper-envelope closure, and a negative Sobolev energy estimate.
\end{abstract}

\section{Introduction and main results}\label{sec:intro}

\subsection{Motivation} \label{subsec:motivation}

The approximation of McKean--Vlasov stochastic differential equations (MVSDEs) by interacting particle systems has attracted sustained attention in recent years, motivated both by theoretical questions in mean field limit theory and by applications in a variety of models. McKean--Vlasov diffusions and the associated nonlinear Fokker--Planck equations serve as canonical mean-field models in probability, statistical physics, and applications. A classical and widely used framework is the \emph{mean-field interacting particle system}, where $N$ particles interact symmetrically through the empirical measure of the entire population. This \emph{exchangeable} setting has been extensively studied; see, e.g., \cite{fournier2014propagation,jabin2014review,mckeanpropagation,sznitman1991topics}, and even quantitative convergence rates have been obtained by relative entropy and modulated energy methods for systems with singular interacting forces \cite{jabin2018quantitative,lacker2021hierarchies,serfaty2020mean}. 

We first recall the classical exchangeable particle system. Fix $T>0$ and $d\ge 1$. Let $\sigma\in\R^{d\times d}$ be a constant invertible matrix, and let $b:[0,T]\times\R^d\to\R^d$ and $K:\R^d\times\R^d\to\R^d$ be bounded measurable functions. For $\nu\in\mathcal P(\R^d)$, we define the convolution-like operator 
\[
(K \star  \nu)(x):=\int_{\R^d} K(x,y)\,\nu(\ud y). 
\]
Let $(B^i)_{i\ge1}$ be i.i.d.\ $d$-dimensional Brownian motions. Let $(X^i_0)_{i\ge1}$ be $\R^d$-valued random variables, independent of $(B^i)_{i\ge1}$. The classical mean-field interacting particle system is then given by, for $N\ge 1$ and $i=1,\dots,N$,
\begin{equation}\label{eq:Classical_IPS}
	\mathrm{d}X^{i,N}_t
	=\Big(b(t,X^{i,N}_t)+ (K\star \mu^N_t)(X^{i,N}_t)\Big)\,\mathrm{d}t
	+\sigma\,\mathrm{d}B^i_t,
	\qquad
	\mu^N_t:=\frac1N\sum_{j=1}^N \delta_{X^{j,N}_t}. 
\end{equation}
Under standard assumptions, the system \eqref{eq:Classical_IPS} is (label-)exchangeable, and the propagation of chaos theory
shows that as $N\to\infty$, the empirical measure $\mu^N_t$ and the law of a tagged particle converge to the McKean--Vlasov limit:
\begin{equation}\label{eq:limitSDE}
	\mathrm{d}\bar X_t
	=\Big(b(t,\bar X_t)+ (K\star \bar\rho_t)(\bar X_t)\Big)\,\mathrm{d}t
	+\sigma\,\mathrm{d}B_t,
	\qquad
	\bar\rho_t:=\Law(\bar X_t),
\end{equation}
equivalently, $\bar\rho_t$ solves the nonlinear Fokker--Planck equation
\begin{equation}\label{eq:NFP}
\partial_t \bar\rho_t
=\frac12\nabla\!\cdot(\sigma\sigma^\top \nabla \bar\rho_t)
-\nabla\!\cdot\!\Big(\bar\rho_t\big(b(t,\cdot)+K\star \bar\rho_t\big)\Big),
\qquad \bar\rho_{t=0}=\bar\rho_0.
\end{equation}
From the numerical standpoint, however, \eqref{eq:Classical_IPS} has two intertwined drawbacks: 
(i) at each time step, computing $K\star\mu^N_t$ for all particles typically costs $O(N^2)$ operations (for general kernels),
and (ii) the system is intrinsically \emph{$N$-dependent}: changing $N$ changes the dynamics of \emph{every} particle through $\mu^N_t$,
so improving accuracy by enlarging $N$ generally forces a full re-simulation of the entire $N$-particle system
(with essentially the same $O(N^2)$ cost profile); see, e.g., classical particle methods \cite{bossy1997stochastic}
and modern linear-cost alternatives such as random-batch strategies \cite{jin2020random}.

\medskip

Motivated by these algorithmic imperatives, we study a different particle approximation:
a \emph{sequential (directed) interacting diffusion} in which the interaction is causal along the index.
Namely, the $i$-th particle interacts only with its predecessors $\{1,\dots,i-1\}$:
\begin{equation}\label{eq:seqSDE}
	\left\{
	\begin{aligned}
		\mathrm{d}X^{1}_t &= b(t,X^{1}_t)\,\mathrm{d}t+\sigma\,\mathrm{d}B^{1}_t,\\
		\mathrm{d}X^{i}_t
		&=\Big(b(t,X^{i}_t)+ (K\star \mu^{i-1}_t)(X^{i}_t)\Big)\,\mathrm{d}t
		+\sigma\,\mathrm{d}B^{i}_t,\qquad i = 2, 3, \cdots, N,\\
		\mu^{i-1}_t&:=\frac{1}{i-1}\sum_{j<i}\delta_{X^{j}_t}.
	\end{aligned}
	\right.
\end{equation}
As $N\to\infty$, the empirical measure (and also in particular the terminal particle $X^N$) associated with \eqref{eq:seqSDE} converges to the \emph{same} McKean--Vlasov equation \eqref{eq:limitSDE};
the key difference is structural: \eqref{eq:seqSDE} is \emph{non-exchangeable}, i.e. the particle system \eqref{eq:seqSDE} is not invariant under index permutations and its interaction graph is \emph{directed/lower-triangular} rather than symmetric.
Thus, the sequential system provides a concrete and analytically tractable example of a non-exchangeable particle system with a sequential (directed) interaction architecture.
A similar lower-triangular, predecessor-dependent structure appears in recent mathematical models of causally masked Transformer self-attention, where token representations are treated as interacting particles driven by prefix empirical inputs; see, for instance, \cite{geshkovski2023emergence,karagodin2024clustering,castin2025unified,rigollet2025meanfield,duerinckx2026kinetic}.

\medskip 

\noindent\textbf{Cost profile and ``online'' refinement.}
For a general kernel, simulating \eqref{eq:seqSDE} up to index $N$ still requires evaluating interactions
against $\mu^{i-1}_t$ for each $i$, so the total interaction work scales as $\sum_{i=2}^N O(i-1) = O(N^2/2)$ 
to be compared with $O(N^2)$ for the fully-coupled mean-field system \eqref{eq:Classical_IPS} (the same time-discretization factor multiplies both costs).
The crucial advantage is \emph{marginal}: once $(X^1,\dots,X^N)$ has been simulated,
adding one more particle $X^{N+1}$ leaves the already-generated trajectories unchanged,
and the extra computational cost is only $O(N)$ (again up to the time-discretization factor).
In contrast, for the classical mean-field system \eqref{eq:Classical_IPS},
increasing $N$ changes the drift of every particle through $\mu^N_t$,
so improving accuracy by enlarging the ensemble typically entails re-computing a new $(N+1)$-particle system,
with extra cost again of order $O(N^2)$.
This ``online'' refinement mechanism is one motivation for the sequential model and is illustrated in Figure~\ref{fig:complexity_comparison}.

\begin{figure}[htbp]
	\centering
	\begin{tikzpicture}[node distance=1.5cm, 
		dot/.style={draw, circle, fill=blue!10, minimum size=8mm}]

		\node[dot] (n1) at (0,1.5) {$X^1$};
		\node[dot] (n2) at (2,1.5) {$X^2$};
		\node[dot] (n3) at (2,0) {$X^3$};
		\node[dot] (n4) at (0,0) {$X^4$};

		\begin{scope}[<->, >=Stealth, thick]
			\draw (n1) -- (n2); \draw (n1) -- (n3); \draw (n1) -- (n4);
			\draw (n2) -- (n3); \draw (n2) -- (n4); \draw (n3) -- (n4);
		\end{scope}

		\node[below=0.5cm of n4, xshift=1cm, font=\small\bfseries] {Classical IPS: $O(N^2)$};
		\node[below=0.9cm of n4, xshift=1cm, font=\footnotesize] {(Full Batch Interaction)};
	\end{tikzpicture}
	\hspace{1.5cm}
		\begin{tikzpicture}[node distance=1.2cm, 
		particle/.style={draw, circle, fill=green!10, minimum size=8mm},
		cache/.style={draw, rectangle, fill=orange!20, minimum size=10mm, dashed}]

		\node[particle] (p1) {$X^1, .., X^{i-1}$}; 
		\node[cache, right=0.7cm of p1] (mu) {$\mu^{i-1}$};
		\node[particle, right=0.7cm of mu] (pi) {$X^i$};
		\node[particle, below=1.0cm of pi] (pi1) {$X^{i+1}$};

		\draw[->, >=Stealth, thick] (p1) -- (mu);
		\draw[->, >=Stealth, thick, blue] (mu) -- (pi) node[midway, above, font=\tiny] {input};
		\draw[->, >=Stealth, thick, red] (pi.north) .. controls +(up:10mm) and +(up:10mm) .. (mu.north) 
		node[midway, above, font=\tiny] {update};
		\draw[->, >=Stealth, thick, blue] (mu.south) .. controls +(down:8mm) and +(left:5mm) .. (pi1.west) 
		node[pos=0.4, left, font=\tiny] {input};

		\node[below=2.6cm of mu, font=\small\bfseries] {Sequential IPS: online extension cost \(O(N)\)};
		\node[below=3.0cm of mu, font=\footnotesize] {(Recursive Online Update)};
	\end{tikzpicture}

	\caption{Left: the classical fully coupled system has a symmetric all-to-all interaction graph. Right: the sequential system has a directed lower-triangular architecture, with particle $i$ driven by the empirical measure $\mu^{i-1}$ of its predecessors.}
	\label{fig:complexity_comparison}
\end{figure}

\medskip

\noindent\textbf{Relation to prior work and contribution.}
Quantitative propagation of chaos for sequential particle approximations was recently studied by Du--Jiang--Li~\cite{du2023sequential}. Their work treats recursive schemes with general step-size or weighting sequences and proves quantitative convergence in Wasserstein-type metrics. The present article focuses instead on path-space relative entropy for the lower-triangular system \eqref{eq:seqSDE}. This leads to a genuinely non-uniform viewpoint along the particle index: rather than controlling only a global or averaged error, we control the one-step information cost of adding particle \(i\) given its predecessors.

More broadly, non-exchangeable mean-field limits and entropy methods have been developed for graph, graphon, and heterogeneous interaction structures; see, e.g., \cite{delattre2016note,oliveira2019interacting,lucon2020quenched,bet2020weakly,bayraktar2020graphon,jabin2021mean,jabin2024dense,lacker2024quantitative,zhou2025non}. In particular, Lacker--Yeung--Zhou~\cite{lacker2024quantitative} develop a general subset-entropy framework for non-exchangeable diffusions. The present paper takes a different direction tailored to the lower-triangular setting: we exploit the exact causality of \eqref{eq:seqSDE} to obtain a direct increment-by-increment entropy decomposition and the sharp pointwise scale \(R_i(T)\lesssim (i-1)^{-1}\). When specialized to lower-triangular sequential matrices, general non-exchangeable estimates recover propagation-of-chaos information at a more aggregate scale; the additional information here is the particle-by-particle entropy profile, together with its tail-block and empirical-measure consequences.

At the fluctuation level, Shkolnikov and Yeung~\cite{shkolnikovyeung2026universal} recently proved a universal central limit theorem for non-exchangeable diffusions with matrix-valued interaction strengths under structural denseness and column-sum assumptions; the limiting Gaussian SPDE is the same as in the classical exchangeable mean-field case \cite{wang2023gaussian}. The lower-triangular model considered here lies in a different regime from this dense-matrix universality theory. For the uniform lower-triangular weights
\[
        \xi_{ij}=\frac{\mathbf 1_{\{j<i\}}}{i-1},
\]
the column sums diverge logarithmically with the system size, so the model is not covered by frameworks requiring uniform column-sum control. Moreover, at the central-limit scale, the sequential structure remains visible: the Gaussian limit is different from the classical exchangeable case; see the follow-up work~\cite{WangZhaoSequentialFluctuations}.

The mathematical challenge is that the sequential system is non-exchangeable: the loss of exchangeability breaks the usual tagged-particle symmetry, while the directed interaction creates a non-uniform dependence profile along the index. The main point of the paper is that this non-uniformity can be measured sharply by incremental relative entropy. Our main results are twofold:
\begin{enumerate}
\item \emph{Quantitative incremental relative entropy estimates.} We prove that the incremental relative entropy of the \(i\)-th particle decays at the sharp rate \((i-1)^{-1}\). Summing along the index yields a global path-space entropy bound of logarithmic order \(\log N\), and the same estimates imply quantitative propagation of chaos for tail blocks.
\item \emph{Sharp empirical convergence.} We prove that the empirical measure of the whole sequential system converges to the McKean--Vlasov law at the canonical \(N^{-1/2}\) scale in negative Sobolev norms. This estimate removes the logarithmic loss that one would obtain by using only the global entropy bound.
\end{enumerate}

\subsection{Main results}

In this article we focus on a simple but intrinsically \emph{non-exchangeable} model  \eqref{eq:seqSDE} with a sequential  interaction structure: particle $i$ interacts only with its predecessors $\{1,\dots,i-1\}$. Since the system is not symmetric in the labels, the classical ``fix a tagged particle'' intuition from exchangeable mean-field models is no longer directly applicable. On the other hand, the loss of symmetry comes with a different and exploitable feature: the dependence is \emph{causal} along the particle index and admits a natural one-step decomposition. Our aim is to quantify, in a genuinely non-uniform way, how the \emph{tail} particles become chaotic (asymptotically independent) as the index grows, and how the law of each tail particle approaches the usual McKean--Vlasov limit, with estimates formulated at the level of \emph{path laws} on $[0,T]$.

\medskip

Throughout, we work on a fixed time horizon $[0,T]$ in dimension $d\ge1$, with a constant non-degenerate diffusion matrix $\sigma\in\R^{d\times d}$, a bounded measurable drift field $b:[0,T]\times\R^d\to\R^d$, and a bounded measurable interaction kernel $K:\R^d\times\R^d\to\R^d$, acting on measures through the usual convolution-type operator $K\star\nu$. The sequential system \eqref{eq:seqSDE} introduced in Subsection~\ref{subsec:motivation} is driven by an i.i.d.\ family of $d$-dimensional Brownian motions $(B^i)_{i\ge1}$ and by an initial array $(X^i_0)_{i\ge1}$ independent of $(B^i)_{i\ge1}$. Importantly, we do \emph{not} require $(X^i_0)_{i\ge1}$ to be independent: our quantitative bounds keep track of possible initial correlations through \emph{incremental} relative entropies at time $0$ (see Theorem~\ref{thm:incremental}). The associated mean-field limit is the standard McKean--Vlasov diffusion  \eqref{eq:limitSDE} with coefficients $(b,K,\sigma)$ and marginal law $(\bar\rho_t)_{t\in[0,T]}$ that solves the nonlinear Fokker--Planck equation \eqref{eq:NFP}.   For bounded measurable \(b\) and \(K\) and  non-degenerate constant diffusion,   the finite sequential system is well defined by induction from the standard  well-posedness   theory for SDEs with bounded measurable drift; see, for instance, \cite{hao2022strong}. We fix throughout the corresponding  McKean--Vlasov   solution \(\bar X\) of \eqref{eq:limitSDE}, whose law is denoted by \((\bar\rho_t)_{t\in[0,T]}\) .

\medskip

 \medskip

Our comparison is made on the path space. For $t\in[0,T]$ set $\mathcal{C}_t:=C([0,t];\R^d)$ and denote
\[
\bar P_{[0,t]}:=\Law(\bar X_{[0,t]})\in\mathcal P(\mathcal{C}_t).
\]
For $i\ge1$, let
\[
P^{1:i}_{[0,t]}:=\Law\big((X^1,\dots,X^i)_{[0,t]}\big)\in\mathcal P(\mathcal{C}_t^i).
\]
For simplicity, we also write that $(X^1,\dots,X^i)_{[0,t]} = X^{1: i}_{[0, t]}$. 
For two probability measures $\mu,\nu$ on the same measurable space $Q$, the relative entropy (also called Kullback–Leibler divergence) is defined by 
\[
\Ent(\mu\mid\nu):=\int_Q  \log\!\Big(\frac{\ud \mu}{\ud \nu}\Big)\, \ud \mu \quad\text{if }\mu\ll\nu,
\qquad \text{and } \Ent(\mu\mid\nu):=+\infty\text{ otherwise}.
\]
The sequential structure suggests measuring convergence \emph{incrementally}. For $i\ge1$ define
\begin{equation}\label{eq:Ri-def}
	R_i(t)
	:=\Ent\!\left(P^{1:i}_{[0,t]}\,\middle|\,P^{1:i-1}_{[0,t]}\otimes \bar P_{[0,t]}\right),
\end{equation}
with the obvious convention that $P^{1:0}_{[0,t]}$ is the unit mass on the empty path space so that  $R_1(t)=\Ent(P^{1}_{[0,t]}\mid \bar P_{[0,t]})$.
Thus $R_i(t)$ quantifies, given the predecessor paths, how far the $i$-th trajectory is from an independent copy of the limit trajectory.  Equivalently, by disintegrating \(P^{1:i}_{[0,t]}\) with respect to its first \(i-1\) coordinates, \(R_i(t)\) is the average conditional entropy of the \(i\)-th path law relative to \(\bar P_{[0,t]}\).  We also write the global entropy as 
\begin{equation}\label{eq:SN-def}
	S_N(t):=\Ent\!\left(P^{1:N}_{[0,t]}\,\middle|\,\bar P_{[0,t]}^{\otimes N}\right).
\end{equation}
A standard chain rule yields $S_N(t)=\sum_{i=1}^N R_i(t)$. A more  precise statement and its proof will be given in  Corollary~\ref{cor:incremental}. 

\medskip

Our main result shows that the incremental entropies decay at the optimal scale $1/(i-1)$.

\begin{theorem}\label{thm:incremental}
	Assume that  the diffusion matrix $\sigma$ is invertible and that the drift term $b$ and the interaction $K$ are bounded measurable.
	For $i\ge2$ set
	  \[
	M_{0,i}
	:=
	R_1(0)\vee
	\max_{2\le j\le i}(j-1)R_j(0).
	\] 
	 \[
	R_j(0)
	=
	\Ent\!\left(
	\Law(X_0^{1:j})
	\,\middle|\,
	\Law(X_0^{1:j-1})\otimes\bar\rho_0
	\right),
	\qquad j\ge2.
	\]
	 Then there exists $C_T>0$, depending only on   \(d,T,\|\sigma^{-1}\|,\|b\|_\infty\) and \(\|K\|_\infty\) ,
	 such that for all $t\in[0,T]$ and all $i\ge2$,
	\[
	R_i(t)\le \frac{C_T}{i-1}\,\bigl(1+M_{0,i}\bigr).
	\]
	In particular, if $M_0:=\sup_{i\ge2}M_{0,i}<\infty$, then $R_i(t)\le \frac{C_T}{i-1}(1+M_0)$ uniformly for  $i \geq 2$.  
\end{theorem}
\begin{remark}\label{rem:M0}
	The quantity $M_0$ controls the size of the time-zero incremental entropies at the natural scale $(i-1)^{-1}$.
	In particular, if $(X_0^i)_{i\ge1}$ are i.i.d.\ with law $\bar{\rho}_0$, then each increment vanishes and hence $M_0=0$.
\end{remark}
\medskip
Summing the tail bounds from Theorem~\ref{thm:incremental} yields a $\log N$ growth for the global entropy $S_N(T)$, but with a coefficient depending on $M_0$.
For the global entropy, one can avoid $M_0$ by summing the integral inequality \eqref{eq:Ri-integral-ineq} and applying a discrete Hardy inequality, leading to a bound depending only on the initial global entropy.
\begin{corollary}\label{cor:global-entropy}
	Assume that the diffusion matrix $\sigma$ is invertible and that the drift term $b$ and the interaction $K$ are bounded measurable. 
	Then for every $N\ge2$,
	\[
	S_N(T)
	\le C_T\Big(\,\Ent\!\left(\Law(X^{1:N}_0)\,\middle|\,\bar\rho_0^{\otimes N}\right)
	+ \,\log N \Big),
	\]
	where $C_T>0$ depends only on   \(d,T,\|\sigma^{-1}\|,\|b\|_\infty\) and \(\|K\|_\infty\) .
	In particular, if we further assume that $(X^i_0)_{i\ge1}$ are i.i.d.\ with law $\bar\rho_0$, then $S_N(T)\le C_T\log N$.
\end{corollary}

\begin{remark}[Sharpness of Theorem \ref{thm:incremental}] \label{rem:sharpness}

	Assume for simplicity that $(X_0^i)_{i\ge1}$ are i.i.d.\ with law $\bar\rho_0$, so that $R_i(0)=0$.
	The Girsanov identity \eqref{eq:Ri-integral} then reads
	  \[
	R_i(T)=\frac12\,\E\int_0^T \big|\sigma^{-1}\Delta_s^i\big|^2\, \ud s,
	\qquad
	\Delta_s^i=\Kop{\mu_s^{i-1}}{X_s^i}-\Kop{\bar\rho_s}{X_s^i}.
	\] 
	To assess the best possible rate dictated by sampling, consider the idealized i.i.d.\ benchmark in which the
	input empirical measure is formed by $(i-1)$ independent copies of the limit process.
	Let $\bar X$ solve the McKean--Vlasov SDE and let $(\bar X^{j})_{j\ge1}$ be i.i.d.\ copies of $\bar X$, independent of $\bar X$.
	Set $\hat\mu_s^{i-1}:=\frac1{i-1}\sum_{j=1}^{i-1}\delta_{\bar X_s^{j}}$ and
	  \[
	\hat\Delta_s^{i}:=\Kop{\hat\mu_s^{i-1}}{\bar X_s}-\Kop{\bar\rho_s}{\bar X_s}.
	\] 
	Conditioning on $\bar X_s$, the random variables $K(\bar X_s,\bar X_s^{j})$ are i.i.d.\ with conditional mean
	$\E[K(\bar X_s,\bar X_s')\mid \bar X_s]$, where $\bar X'$ is an independent copy of $\bar X$.
	By the usual \(L^2\)-type law of large numbers computations, we obtain the exact identity
	  \[
	\E\big[|\hat\Delta_s^{i}|^2\big]	=\frac{1}{i-1}\, 	\E\Big[\big|
	K(\bar X_s,\bar X_s')-\Kop{\bar\rho_s}{\bar X_s}
	\big|^2\Big].
	\] 
	Therefore the corresponding ``increment energy'' satisfies the exact scaling
	\[
	\hat R_i(T)
	:=\frac12\,\E\int_0^T \big|\sigma^{-1}\hat\Delta_s^{i}\big|^2\, \ud s
	=\frac{1}{2(i-1)}\int_0^T
	\E\Big[\big|\sigma^{-1}\big(\tilde K_s-\E[\tilde K_s\mid \bar X_s]\big)\big|^2\Big]\, \ud s,
	\]
	with $\tilde K_s:=K(\bar X_s,\bar X_s')$.
	  
	Thus
	\[
	\widehat R_i(T)=\frac{C_{\rm iid}}{i-1},
	\qquad
	C_{\rm iid}:=\frac12\int_0^T
	\E\left[
	\left|\sigma^{-1}\left(\tilde K_s-\E[\tilde K_s\mid \bar X_s]\right)\right|^2
	\right]\,ds .
	\]
	In particular, whenever \(C_{\rm iid}>0\), the idealized i.i.d. benchmark has exactly the order \((i-1)^{-1}\).  This shows that the   \(1/(i-1)\)  order is the canonical   sampling barrier in the non-degenerate case,  even under the most favorable i.i.d.  input, and thus provides a natural benchmark for the sharpness of our incremental entropy bound.

	\smallskip

	Finally, we emphasize that the logarithmic growth of the \emph{global} entropy in our sequential model reflects its
	triangular structure: particle $i$ effectively averages over only $(i-1)$ predecessors, so summing the canonical scale
	$(i-1)^{-1}$ over $i=2,\dots,N$ yields $\sum_{i=2}^N (i-1)^{-1}\sim \log N$.
	This mechanism is fundamentally different from the classical exchangeable mean-field system, where each particle interacts
	with the empirical measure of all $N$ particles and global entropy bounds of order $O(1)$ are available; see, e.g.,
	\cite{jabin2018quantitative}.

\end{remark}

\medskip

The directed (triangular) dependence implies a genuinely non-uniform approximation along the index:
early particles may be strongly correlated, while the tail particles should become increasingly well averaged,
since particle $i$ only sees $(i-1)$ predecessors through $\mu_t^{i-1}$. The following corollary gives quantitative  \emph{tail propagation of chaos}  on path space by controlling the tail block by the last incremental entropies.

\begin{corollary}\label{cor:tail-poc}
Under the assumptions of Theorem  ~\ref{thm:incremental},  assume further that the initial random variables   \((X^i_0)_{i\ge1}\)  are i.i.d.\ with common law   \(\bar\rho_0\). Let
\[
Y^{N,m}:=\big(X^{N-m+1}_{[0,T]},\ldots,X^N_{[0,T]}\big)\in C([0,T];\mathbb R^d)^m.
\] 
Then, for \(1\le m\le N\),
\[
\Ent\!\left(
\Law(Y^{N,m})
\,\middle|\,
\bar P_{[0,T]}^{\otimes m}
\right)
\le
\sum_{k=N-m+1}^{N} R_k(T).
\] 
In particular, if   \(m\le N/2\), then
\begin{equation}\label{eq:tail-entropy}
\Ent\!\left(
\Law(Y^{N,m})
\,\middle|\,
\bar P_{[0,T]}^{\otimes m}
\right)
\le
C_T\,\frac{m}{N}.
\end{equation} 
\end{corollary}

\begin{remark}\label{rem:tail-tv}
	By Pinsker's inequality,
	\[
	\big\|\Law(Y^{N,m})-\bar P_{[0,T]}^{\otimes m}\big\|_{\TV}
	\le \sqrt{2\,\Ent\!\left(\Law(Y^{N,m})\,\middle|\,\bar P_{[0,T]}^{\otimes m}\right)}.
	\]
	Combining with \eqref{eq:tail-entropy} yields
	$\big\|\Law(Y^{N,m})-\bar P_{[0,T]}^{\otimes m}\big\|_{\TV}\lesssim \sqrt{m/N}$, and in particular
	$\|\Law(X^N_{[0,T]})-\bar P_{[0,T]}\|_{\TV}\lesssim N^{-1/2}$.
\end{remark}

\medskip
The previous results quantify propagation of chaos at the level of fixed (or finitely many) tail particle paths. 
We now turn to the collective behavior of the system as a whole, described by the empirical measure process
\[
\mu_t^N=\frac{1}{N}\sum_{i=1}^N \delta_{X_t^i}, \qquad t\in[0,T].
\]
 For mean-field systems, quantitative control of $\mu^N-\bar\rho$   can  be obtained from global relative entropy bounds; see \cite{wang2023gaussian}. In the sequential setting, we   instead  exploit the triangular   particle structure directly  and work in negative Sobolev norms,   which yields  the optimal $N^{-1/2}$ scale.
  The next theorem uses an additional regularity assumption on the limiting velocity field \(V=b+K\star\bar\rho\). This assumption is not used in the entropy argument of Section~\ref{sec:incremental-proof}. Rather, it is necessary for the particle-system central limit theorem, as explained in \cite{wang2023gaussian}.
 \begin{theorem}\label{thm:empirical}

 Assume that the diffusion matrix   \(\sigma\)  is invertible and that   \(b\) and \(K\) are bounded measurable.  Suppose further that the initial random variables   \((X^i_0)_{i\ge1}\)  are i.i.d.\ with common law   \(\bar\rho_0\). Fix \(\beta>d/2+2\), and let \(m\in\mathbb N\) be any integer with \(m>\beta\), and define the limiting velocity field
\[
V_t(x):=b(t,x)+\Kop{\bar\rho_t}{x}.
\] 
 Assume that
\begin{equation}\label{eq:assump-v-reg}
V_T:=\sup_{t\in[0,T]}\|V_t\|_{W^{m,\infty}}<\infty.
\end{equation}
 Then there exists a constant   \(C_T>0\) , depending only on
  \[
\beta,\ T,\ \|\sigma\|_{\op},\ \|\sigma^{-1}\|_{\op},\ \|b\|_\infty,\ \|K\|_\infty,\ \text{and }V_T,
\]
 such that for all   \(N\ge2\),
\begin{equation}\label{eq:empirical-rate}
\E\Big[
\sup_{0\le t\le T}
\|\mu_t^N-\bar\rho_t\|_{H^{-\beta}}^2
\Big]
+
\E\int_0^T
\|\mu_t^N-\bar\rho_t\|_{H^{-\beta+1}}^2\,\mathrm dt
\le
\frac{C_T}{N}.
\end{equation} 

 \end{theorem}

The i.i.d.\ initial condition in Theorem~\ref{thm:empirical} is used to keep the empirical-measure estimate focused on the sequential interaction error: it gives \(R_i(0)=0\) for \(i\ge2\) and the standard \(N^{-1/2}\) initial empirical fluctuation bound in \(H^{-\beta}\). Extensions to correlated initial arrays should be possible under corresponding initial empirical fluctuation and incremental entropy assumptions, but we do not pursue this here.

 \begin{remark}\label{rem:empirical-why}
	The Donsker--Varadhan variational  formula combined with the global relative entropy estimate
	$S_N(T)\lesssim \log N$, as in \cite{wang2023gaussian},
	typically yield at best a rate of order $\sqrt{\log N/N}$ for the convergence of empirical measures in $H^{-s}$.

	The estimate \eqref{eq:empirical-rate} removes this logarithmic loss.
		The key point is that we do not rely on $S_N(T)$ alone: instead, we exploit the sequential structure
		and analyze the Fokker--Planck equation in negative Sobolev norms.
		This yields the canonical $N^{-1/2}$ scale, which is optimal for empirical averages.
\end{remark}

  \begin{remark}[Regularity assumption on the limiting velocity field \(V=b+K\star\bar\rho\)] \label{rem:v-regularity}

Theorem~\ref{thm:empirical} assumes  a Sobolev-multiplier   bound on
\[
V_t=b(t,\cdot)+\Kop{\bar\rho_t}.
\]
This  is sufficient for the   negative Sobolev energy estimate and is independent of the entropy argument in Section~\ref{sec:incremental-proof}. For example, the assumption holds if \(b\in L^\infty([0,T];W^{m,\infty}(\mathbb R^d))\) and \(\Kop{\bar\rho_t}\) is uniformly bounded in \(W^{m,\infty}\).  In the translation-invariant   case \(K(x,y)=K_0(x-y)\), one has \(\Kop{\bar\rho_t}=K_0*\bar\rho_t\); if \(\bar\rho_t\) admits a smooth density and
\[
\sup_{t\le T}\sum_{|\alpha|\le m}\|\partial^\alpha\bar\rho_t\|_{L^1}<\infty,
\]
then \(K_0*\bar\rho_t\in W^{m,\infty}\) whenever \(K_0\in L^\infty\).

 \end{remark}

  \paragraph{Proof strategy.}
 We view the joint law as being built by adjoining one trajectory at a time. This gives the exact chain rule \(S_N(t)=\sum_{i=1}^N R_i(t)\). The core of the proof is then to control each increment directly. First, a Girsanov identity expresses \(R_i(t)\) through the quadratic energy of the drift mismatch. Second, to reveal averaging along the predecessors, we replace the predecessor empirical measure by an averaged conditional measure, following the conditional-measure viewpoint used in \cite{wang2022mean}. This separates a martingale-difference term from a predictable term. The martingale term is controlled by exponential-integrability and the Donsker--Varadhan variational formula under a decoupled reference law, while the predictable term is bounded in terms of previous increments by Pinsker's inequality. Finally, an upper-envelope closure and a Hardy inequality yield the sharp \(1/(i-1)\) decay, the logarithmic global entropy bound, and tail-block propagation of chaos.

For the empirical-measure process, we do not rely on the global entropy bound, which would introduce an extraneous \(\sqrt{\log N}\) loss. Instead, we write the \(H^{-\beta}\)-norm as a Bessel-kernel energy and follow its evolution along the sequential dynamics. The interaction-error term is controlled by the incremental entropy estimate, and the martingale contribution is absorbed through the dissipative part of the negative Sobolev energy inequality.

\paragraph{Organization of the article.}
Section~\ref{sec:prelim} collects the information-theoretic and Sobolev tools used throughout the paper. Section~\ref{sec:incremental-proof} proves the incremental relative entropy estimate, derives the global entropy and tail-chaos consequences, and includes a discussion of general sequential weights. Section~\ref{sec:consequences} proves the sharp empirical-measure convergence estimate in negative Sobolev norms.

\section{Preliminaries}\label{sec:prelim}

This section collects notation and standard information-theoretic tools used repeatedly in the proof of
Theorem~\ref{thm:incremental}. We also recall a Girsanov--entropy identity for diffusions with the same
non-degenerate constant diffusion matrix, which turns a comparison of path laws into a quadratic energy of the
drift difference.

\smallskip
\noindent\emph{Notation and conventions.}  Fix $T>0$ and $d\ge 1$. For $t\in[0,T]$, let $\mathcal C_t:=C([0,t];\R^d)$ endowed with its Borel $\sigma$-field, and set $\mathcal C_t^i:=(\mathcal C_t)^i$.
For a measurable map $T$ and a measure $\mu$, we write $T_\#\mu$ for the push-forward  measure of $\mu$ under $T$.
For a finite signed measure $\nu$ on $\R^d$ and a test function $\varphi$, we use the notation
\[
\langle \nu,\varphi\rangle:=\int_{\R^d}\varphi\,\mathrm d\nu.
\]
For probability measures $\mu$ and $\nu$ on a common measurable space, $\Ent(\mu\mid\nu)$ denotes the relative entropy, with the convention that $\Ent(\mu\mid\nu)=+\infty$ if $\mu\not\ll\nu$.
We denote by $\|\cdot\|_{\TV}$ the total variation distance.
We also set $\|K\|_\infty:=\|K\|_{L^\infty(\R^d\times\R^d)}$, and denote by $\|\sigma^{-1}\|$ the operator norm of $\sigma^{-1}$.
Throughout, $C$ denotes a finite positive constant whose value may change from line to line.

 For $s\in\R$, we denote by $H^s(\R^d)$ the Sobolev space with norm
\[
\|f\|_{H^s}^2:=\int_{\R^d}(1+|\xi|^2)^s\,|\hat f(\xi)|^2\ud \xi,
\]
where $\hat f$ denotes the Fourier transform of the function $f$. 
The dual space of $H^s$ is denoted by $H^{-s}$, and for $\nu\in H^{-s}$ we use the dual norm
\begin{equation}\label{eq:Hminus-def-prelim}
	\|\nu\|_{H^{-s}}
	:=\sup_{\|\psi\|_{H^s}\le1}\langle \nu,\psi\rangle .
\end{equation}
When $s>d/2$, we define the Bessel kernel $G_s:=(1-\Delta)^{-s}\delta_0$ by
\[
\widehat{G_s}(\xi)=(1+|\xi|^2)^{-s}.
\]
Since $(1+|\xi|^2)^{-s}\in L^1(\R^d)$ for $s>d/2$, the inverse Fourier transform yields
$G_s\in C_b(\R^d)$ (in particular $G_s$ is continuous and bounded).
Moreover, if $s>d/2+1$, then $|\xi|^2(1+|\xi|^2)^{-s}\in L^1(\R^d)$, hence $G_s\in C_b^2(\R^d)$.

For $s>d/2$, the $H^{-s}$-norm admits the kernel representation
\begin{equation}\label{eq:Hminus-kernel-prelim}
	\|\nu\|_{H^{-s}}^2
	=\iint_{\R^d\times\R^d} G_s(x-y)\,\nu(\ud x)\nu( \ud y).
\end{equation}
Finally, when $s>d/2+1$, we have the Sobolev embedding $H^s(\R^d)\hookrightarrow W^{1,\infty}(\R^d)$, namely
\begin{equation}\label{eq:Sobolev-embed-prelim}
	\|\psi\|_\infty+\|\nabla\psi\|_\infty\le C_s\,\|\psi\|_{H^s}.
\end{equation}

  We shall use the standard  It\^o   isometry and  Burkholder--Davis--Gundy   inequality for the real-valued martingales appearing in Section~\ref{sec:consequences}.

 \begin{lemma}\label{lem:multiplier}
	Let $\beta>0$ and let \(m\in\mathbb N\) satisfy \(m>\beta\). There exists $C_\beta>0$ such that for all
	$f\in W^{m,\infty}(\R^d)$ and $g\in H^\beta(\R^d)$,
	\[
	\|fg\|_{H^\beta}\le C_\beta \|f\|_{W^{m,\infty}}\|g\|_{H^\beta}.
	\]
	Consequently, if $\Phi\in H^{\beta+1}(\R^d)$ and $v\in W^{m,\infty}(\R^d;\R^d)$, then
	\[
	\|v\cdot\nabla\Phi\|_{H^\beta}\le C_\beta \|v\|_{W^{m,\infty}}\|\Phi\|_{H^{\beta+1}}.
	\]
\end{lemma}

\begin{proof}
	This is a standard Moser/Kato--Ponce type estimate: multiplication by a $C^m_b$ function
	is a bounded operator on $H^\beta$ when $m>\beta$.  
\end{proof}

For $\rho\in\mathcal P(\R^d)$, we write $\Kop{\rho}{x}:=\int_{\R^d}K(x,y)\,\rho(\ud y)$.
For $i\ge2$ we use the predecessor empirical measure
$\mu_t^{i-1}:=\frac1{i-1}\sum_{j<i}\delta_{X_t^j}$, and we write
$\mu_t^N:=\frac1N\sum_{i=1}^N\delta_{X_t^i}$ for the global empirical measure.

On path space we denote
\[
\bar P_{[0,t]}:=\Law(\bar X_{[0,t]})\in\mathcal P(\mathcal C_t),
\qquad
P^{1:i}_{[0,t]}:=\Law((X^1,\dots,X^i)_{[0,t]})\in\mathcal P(\mathcal C_t^i).
\]
We recall the incremental and global relative entropies
\[
R_i(t):=\Ent\!\left(P^{1:i}_{[0,t]}\,\middle|\,P^{1:i-1}_{[0,t]}\otimes \bar P_{[0,t]}\right),
\qquad
S_N(t):=\Ent\!\left(P^{1:N}_{[0,t]}\,\middle|\,\bar P_{[0,t]}^{\otimes N}\right).
\]
Note that $\mathcal C_0\simeq \R^d$. In particular, at $t=0$ the quantities $R_i(0)$ and $S_N(0)$ are
\emph{initial} (non-path) relative entropies:
\[
R_i(0)=\Ent\!\left(\Law(X_0^{1:i})\,\middle|\,\Law(X_0^{1:i-1})\otimes \bar\rho_0\right),
\qquad
S_N(0)=\Ent\!\left(\Law(X_0^{1:N})\,\middle|\,\bar\rho_0^{\otimes N}\right),
\]
where $\bar\rho_0:=\Law(\bar X_0)$.

For each $t\in[0,T]$ and each $i\ge2$, we denote by
\[
\mathcal G_t^{i-1}:=\sigma(X_t^1,\dots,X_t^{i-1}),
\qquad
\nu_t^i:=\Law(X_t^i\mid \mathcal G_t^{i-1})
\]
a regular conditional distribution (a probability kernel on $\R^d$ given $X_t^{1:i-1}$).
A jointly measurable choice in $(t,\omega)$ (or in $(t,x^{1:i-1})$) will be constructed in
Subsection~\ref{subsec:replacement}, more precisely in Lemma~\ref{lem:time-disintegration}, by disintegrating the
product measure $dt\otimes \Law(X_t^{1:i})$.

For \(i=1\), we use the convention
\[
\mathcal G_t^0:=\{\varnothing,\Omega\},
\qquad
\nu_t^1:=\Law(X_t^1).
\]
This convention is used whenever sums over \(j<i\) include the index \(j=1\).

 \medskip

We collect the Csisz\'ar--Kullback--Pinsker inequality, data processing inequality and the chain rule of relative entropy as follows. 

\begin{lemma}\label{lem:entropy-tools}
	Let $X,Y$ be Polish spaces and let $\mu,\nu\in\mathcal P(X)$. Let $T:X\to Y$ be a measurable map. Then the following properties hold.
	\begin{enumerate}
			\item[(i)] \textbf{(Csisz\'ar--Kullback--Pinsker).} \(\|\mu-\nu\|_{\TV}\le \sqrt{2\,\Ent(\mu\mid\nu)}\).
			\item[(ii)] \textbf{(Data processing inequality).} \(\Ent(T_\#\mu\mid T_\#\nu)\le \Ent(\mu\mid\nu)\).

		\item[(iii)] \textbf{(Chain rule / disintegration).}
		Let $E,F$ be Polish spaces. Let $\mu\in\mathcal P(E\times F)$ and $\pi_E \in\mathcal P(E)$,  $\pi_F \in\mathcal P(F)$ and denote by $\mu_E$ the $E$-marginal of $\mu$.
		Disintegrate $\mu$ with respect to $\mu_E$ as
		\[
		\mu(\ud x, \ud y)=\mu(\ud  y\mid x)\,\mu_E(\ud  x).
		\]
		Then
		\[
		\Ent(\mu\mid \pi_E \otimes \pi_F)
		= \Ent(\mu_E\mid \pi_E ) + \int_E \Ent\!\left(\mu(\cdot\mid x)\mid \pi_F \right)\,\mu_E(\ud x ).
		\]
	\end{enumerate}
\end{lemma}

\begin{proof}
	Items (i)--(ii) are classical.
	For (iii), whenever $\mu\ll \pi_E \otimes \pi_F$, we have
	\[
	\frac{\ud \mu}{\ud (\pi_E \otimes\pi_F)}(x,y)
	= \frac{\ud \mu_E}{\ud \pi_E}(x)\cdot \frac{ \ud \mu(\cdot\mid x)}{\ud \pi_F}(y),
	\qquad (\pi_E \otimes \pi_F)\text{-a.e.},
	\]
	and the identity follows by integrating $\log(\cdot)$ with respect to $\mu$; otherwise both sides are $+\infty$ by convention.
\end{proof}

\smallskip

\begin{corollary}\label{cor:incremental}
	For every $t\in[0,T]$, the following statements hold true. 
	\begin{enumerate}
		\item[(i)] For every \(N\ge1\), \(S_N(t)=\sum_{i=1}^N R_i(t)\).

		\item[(ii)] For  \(m = 1, 2, \dots,N\), \(\Ent\!\left(P^{1:N}_{[0,t]}\,\middle|\,P^{1:N-m}_{[0,t]}\otimes \bar P_{[0,t]}^{\otimes m}\right)=\sum_{j=0}^{m-1} R_{N-j}(t)\).

		\item[(iii)]
		Let $e_t:\mathcal C_t\to\R^d$ be the evaluation map $e_t(\gamma)=\gamma_t$.
		Then for each $i\ge1$,
		\[
		\Ent\!\left(\Law(X_t^i)\,\middle|\,\bar\rho_t\right)\le R_i(t),
		\qquad
		\bar\rho_t:=(e_t)_\#\bar P_{[0,t]}.
		\]
	\end{enumerate}
\end{corollary}

\begin{proof}
	(i) We apply Lemma~\ref{lem:entropy-tools}(iii) with $E=\mathcal C_t^{i-1}$, $F=\mathcal C_t$, 
	$\mu=P^{1:i}_{[0,t]}$, $\pi_E=\bar P_{[0,t]}^{\otimes (i-1)}$, and $\pi_F=\bar P_{[0,t]}$.
	Since the $E$-marginal of $P^{1:i}_{[0,t]}$ is $P^{1:i-1}_{[0,t]}$, we obtain
	\[
	\Ent\!\left(P^{1:i}_{[0,t]}\,\middle|\,\bar P_{[0,t]}^{\otimes i}\right)
	= \Ent\!\left(P^{1:i-1}_{[0,t]}\,\middle|\,\bar P_{[0,t]}^{\otimes (i-1)}\right)
	+ \int_E  \Ent\! \left(P^{1:i}_{[0,t]}(\cdot \vert x)\,\middle|\,\bar P_{[0,t]}\right) \ud P^{1:i-1}_{[0, t]}(x). 
	\]
	Similarly,   applying  Lemma~\ref{lem:entropy-tools}(iii) again but now for the   incremental  relative entropy $R_i(t)$, one has that
	\begin{equation}\label{eq:R_i_conditional}
	\begin{split}
		R_i(t)  & = \Ent\!\left(P^{1:i}_{[0,t]}\,\middle|\, P^{1:i-1}_{[0,t]}\otimes  \bar P_{[0,t]}\right)
		= \Ent\!\left(P^{1:i-1}_{[0,t]}\,\middle|\, P^{1:i-1}_{[0,t]} \right)
		+ \int_E  \Ent\! \left(P^{1:i}_{[0,t]}(\cdot \vert x)\,\middle|\,\bar P_{[0,t]}\right) \ud P^{1:i-1}_{[0, t]}(x) \\
		& = \int_E  \Ent\! \left(P^{1:i}_{[0,t]}(\cdot \vert x)\,\middle|\,\bar P_{[0,t]}\right) \ud P^{1:i-1}_{[0, t]}(x). \\
	\end{split} 
	\end{equation}
	Consequently, one obtains that 
	\begin{equation}\label{eq:RitRepre}
		R_i(t) = \Ent\!\left(P^{1:i}_{[0,t]}\,\middle|\,\bar P_{[0,t]}^{\otimes i}\right)
		-  \Ent\!\left(P^{1:i-1}_{[0,t]}\,\middle|\,\bar P_{[0,t]}^{\otimes (i-1)}\right). 
	\end{equation}
	By a telescoping sum, we obtain
	\[
	S_N(t)
	= \Ent\!\left(P^{1:N}_{[0,t]}\,\middle|\,\bar P_{[0,t]}^{\otimes N}\right)
	= \sum_{i=2}^N \Bigg(
	\Ent\!\left(P^{1:i}_{[0,t]}\,\middle|\,\bar P_{[0,t]}^{\otimes i}\right)
	-\Ent\!\left(P^{1:i-1}_{[0,t]}\,\middle|\,\bar P_{[0,t]}^{\otimes (i-1)}\right)
	\Bigg)
	+\Ent\!\left(P^{1}_{[0,t]}\,\middle|\,\bar P_{[0,t]}\right).
	\]
	In particular, with the equation \eqref{eq:RitRepre} and the definition 
	\[
	R_1(t):=\Ent\!\left(P^{1}_{[0,t]}\,\middle|\,\bar P_{[0,t]}\right),\qquad
	\]
	we have \(S_N(t)=\sum_{i=1}^N R_i(t)\).

	(ii)   Fix \(m\in\{1,\ldots,N\}\). For \(\ell=N-m,\ldots,N\), define
	\[
	E_\ell(t)
	:=
	\Ent\!\left(
	P^{1:\ell}_{[0,t]}
	\,\middle|\,
	P^{1:N-m}_{[0,t]}\otimes \bar P_{[0,t]}^{\otimes(\ell-N+m)}
	\right),
	\] 
	 with the convention \(E_{N-m}(t)=0\). Applying the chain rule of Lemma~\ref{lem:entropy-tools}(iii) to the last coordinate of \(P^{1:\ell}_{[0,t]}\) gives, for every \(\ell=N-m+1,\ldots,N\),
	\[
	E_\ell(t)-E_{\ell-1}(t)
	=
	\int
	\Ent\!\left(
	P^{1:\ell}_{[0,t]}(\mathrm d\gamma^\ell\mid \gamma^{1:\ell-1})
	\,\middle|\,
	\bar P_{[0,t]}
	\right)
	P^{1:\ell-1}_{[0,t]}(\mathrm d\gamma^{1:\ell-1})
	=
	R_\ell(t),
	\]
	 where   the last equality is exactly the disintegrated form of \(R_\ell(t)\), as in \eqref{eq:R_i_conditional}. Summing this identity over \(\ell=N-m+1,\ldots,N\) yields
	\[
	\Ent\!\left(
	P^{1:N}_{[0,t]}
	\,\middle|\,
	P^{1:N-m}_{[0,t]}\otimes \bar P_{[0,t]}^{\otimes m}
	\right)
	=
	\sum_{\ell=N-m+1}^{N} R_\ell(t)
	=
	\sum_{j=0}^{m-1}R_{N-j}(t).
	\]

	(iii) Consider the measurable map $T:\mathcal C_t^i\to\R^d$ defined by $T(\gamma^1,\dots,\gamma^i)=\gamma_t^i$.
	Then $T_\# P^{1:i}_{[0,t]}=\Law(X_t^i)$ and $T_\#(P^{1:i-1}_{[0,t]}\otimes \bar P_{[0,t]})=\bar\rho_t$.
	By the data processing inequality in Lemma~\ref{lem:entropy-tools}(ii), one has $\Ent(\Law(X_t^i)\mid \bar\rho_t)\le R_i(t)$.
\end{proof}

\medskip

We now  state a version of a classical relative entropy estimate for diffusions which follows from Girsanov’s theorem. 

\begin{lemma}\label{lem:girsanov-entropy}
	Let $k\in\mathbb N$ and let $\sigma\in\mathbb R^{k\times k}$ be invertible.
	Let $b^1,b^2:[0,T]\times \mathcal C_T^k\to\mathbb R^k$ be progressively measurable and bounded.
	Let $P^i\in\mathcal P(\mathcal C_T^k)$ denote the path law of the SDE
	\[
	dZ^i_t=b^i(t,Z^i)\, \ud t+\sigma \ud  W_t,\qquad t\in[0,T],
	\]
	with initial laws $P^i_0$, where $i=1, 2$. 
	Assume $\Ent(P^1_0\mid P^2_0)<\infty$.
	Then for every $t\in[0,T]$,
	\[
	\Ent\!\left(P^1_{[0,t]}\,\middle|\,P^2_{[0,t]}\right)
	= \Ent(P^1_0\mid P^2_0)
	+\frac12\,\E_{P^1}\!\int_0^t
	\big|\sigma^{-1}\big(b^1(s,Z^1)-b^2(s,Z^1)\big)\big|^2\ud s.
	\]
\end{lemma}

This lemma follows from a standard application of Girsanov's theorem: under a bounded drift and a non-degenerate constant diffusion matrix, the above path measures are well defined, and the stated identity holds; see, for instance, \cite[Lemma~4.4]{lacker2021hierarchies}.

\section{Proof of the incremental relative entropy bound}\label{sec:incremental-proof}

This section proves Theorem~\ref{thm:incremental}.
The key quantity is the incremental relative entropy
\[
R_i(t)=\Ent\!\left(P^{1:i}_{[0,t]}\,\middle|\,P^{1:i-1}_{[0,t]}\otimes \bar P_{[0,t]}\right),\qquad i\ge2,
\]
which measures how much the $i$-th path deviates from an independent McKean--Vlasov copy when conditioned on the first $(i-1)$ paths.
The proof consists of  three steps.
First, we derive an explicit Girsanov/relative-entropy identity expressing $R_i(t)$ through the quadratic energy of the drift mismatch.
Second, we replace the predecessor empirical measure by a conditional average measure, which separates a martingale-difference term
from a predictable term.
Third, we close the resulting estimate by an upper-envelope argument and Gronwall's inequality.

\medskip
\noindent
Throughout this article, $\|K\|_\infty:=\|K\|_{L^\infty(\R^d\times\R^d)}$ and $\|\sigma^{-1}\|$ denotes the operator norm.
For convenience, set
\begin{equation}\label{eq:kappa-def}
	\kappa:=\|\sigma^{-1}\|^2\,\|K\|_\infty^2 .
\end{equation}

\subsection[Quadratic identity for Ri(t)]{Quadratic identity for $R_i(t)$}\label{subsec:quad-Ri}

The measures $P^{1:i}_{[0,t]}$ and $P^{1:i-1}_{[0,t]}\otimes \bar P_{[0,t]}$ coincide on the first $(i-1)$ coordinates. 
Indeed, in the sequential system the dynamics of $(X^1,\dots,X^{i-1})$ do not involve particle $i$. Hence the two laws differ only through the $i$-th coordinate, namely through its initial distribution and its drift.
This allows a direct application of Girsanov's theorem and yields a quadratic representation of $R_i(t)$ as follows. 

\begin{lemma}  \label{lem:Ri-girsanov}

For each \(i\ge2\) and each \(t\in[0,T]\),
\begin{equation}\label{eq:Ri-integral}
R_i(t)=R_i(0)+\frac12\,\E\int_0^t
\big|\sigma^{-1}\Delta_s^i\big|^2\,\mathrm ds,
\end{equation} 
where
  \[
\Delta_s^i:=\Kop{\mu_s^{i-1}}{X_s^i}-\Kop{\bar\rho_s}{X_s^i},
\qquad
\mu_s^{i-1}:=\frac1{i-1}\sum_{j<i}\delta_{X_s^j}.
\] 
  For \(i=1\), the same argument gives
\[
R_1(t)=R_1(0)+\frac12\,\E\int_0^t
\big|\sigma^{-1}\Kop{\bar\rho_s}{X_s^1}\big|^2\,\mathrm ds.
\]

 \end{lemma}
\begin{remark}
	We emphasize that the ``initial'' quantities $R_i(0)$ and $S_N(0)$ are \emph{not} path-space relative entropies.
	They only depend on the joint law of the initial configuration.
	Indeed, by definition of $R_i(t)$ as a relative entropy on $C([0,t];\R^d)^i$ and by restricting it to time $0$, one has 
	\begin{equation}\label{eq:Ri0-initial}
		R_i(0)
		=\Ent\!\left(\Law(X^{1:i}_0)\,\middle|\,
		\Law(X^{1:i-1}_0)\otimes \bar\rho_0\right),
		\qquad i\ge2.
	\end{equation}
	Similarly, the global relative entropy at time $0$ reduces to the relative entropy of the initial laws:
	\begin{equation}\label{eq:SN0-initial}
		S_N(0)
		=\Ent\!\left(\Law(X^{1:N}_0)\,\middle|\,\bar\rho_0^{\otimes N}\right).
	\end{equation}
	In particular, if $(X_0^i)_{i\ge1}$ are i.i.d.\ with law $\bar\rho_0$, then $R_i(0)=0$ for all $i\ge2$
	and $S_N(0)=0$ for all $N$. In this case, we have the trivial estimate for $R_1(t)$ as 
	\begin{equation}\label{eq:R1-explicit}
		R_1(t)
		= R_1(0)+\frac12\,\E\int_0^t \Big|\sigma^{-1}\Kop{\bar\rho_s}{X_s^1}\Big|^2\ud s
		\le \frac12\,\|\sigma^{-1}\|^2\,\|K\|_\infty^2\,t
		=\frac12\,\kappa\,t .
	\end{equation}

\end{remark}

Even though the proof of the lemma is very straightforward, for   completeness  we give a detailed proof below.

  \begin{proof}[Proof of Lemma \ref{lem:Ri-girsanov}] 
	Fix any integer $i\ge2$  and $t\in[0,T]$, and set
	$Q^{1:i}_{[0,t]}:=P^{1:i-1}_{[0,t]}\otimes \bar P_{[0,t]}$ on $\mathcal C_t^i$.
	By definition, the \(i\)-th incremental relative entropy reads 
	\[
	R_i(t)=\Ent\!\left(P^{1:i}_{[0,t]}\,\middle|\,Q^{1:i}_{[0,t]}\right).
	\]

	Let $Z=(Z^1,\dots,Z^i)$ be the canonical coordinate process on $\mathcal C_t^i$.
	Under the law  $P^{1:i}_{[0,t]}$, the process $Z$ solves the sequential system up to time $t$; under the reference law $Q^{1:i}_{[0,t]}$, 
	the first $(i-1)$ coordinates share the same law, while the $i$-th coordinate  $Z^i$ is an independent McKean--Vlasov copy.
	In particular, under both measures,  the diffusion matrix is $\Sigma=\mathrm{diag}(\sigma,\dots,\sigma)$, and the drift fields
	coincide in the first $(i-1)$ coordinates. In the last coordinate, the drift difference equals
	\[
	\Delta_s^i
	=\Kop{\mu_s^{i-1}}{Z_s^i}-\Kop{\bar\rho_s}{Z_s^i},
	\qquad \mbox{with} \, \, \, 
	\mu_s^{i-1}:=\frac1{i-1}\sum_{j<i}\delta_{Z_s^j}.
	\]
	Therefore applying Lemma~\ref{lem:girsanov-entropy} yields \eqref{eq:Ri-integral}, where  the initial  term is exactly $R_i(0)$.

	The case $i=1$ is identical, except that there is no predecessor empirical measure and the drift mismatch reduces to \(-\Kop{\bar\rho_s}{Z_s^1}\), which gives the displayed formula for \(R_1(t)\).

 \end{proof}

\subsection{Replacement of the predecessor empirical measure}\label{subsec:replacement}

Fix $i\ge2$ and $t\in[0,T]$. Recall the drift mismatch
\[
\Delta_t^i := \Kop{\mu_t^{i-1}}{X_t^i}-\Kop{\bar\rho_t}{X_t^i},
\qquad \mu_t^{i-1}:=\frac1{i-1}\sum_{j<i}\delta_{X_t^j}.
\]
Again we recall $Q^{1:i}_{[0,t]}:=P^{1:i-1}_{[0,t]}\otimes \bar P_{[0,t]}$ on $\mathcal C_t^i$. 

The goal of this subsection is to control $\Delta_t^i$ at the random correlated point $X_t^i$. 
To this end, we insert a predictable mean measure $\bar\nu_t^{i-1}$ constructed from conditional laws, later we call it {\em average conditional measure}, 
and decompose $\Delta_t^i$ by the usual add--subtract trick:
\begin{equation}\label{eq:Delta-decomp}
	\begin{aligned}
		\Delta_t^i
		&= \Kop{\mu_t^{i-1}}{X_t^i}-\Kop{\bar\rho_t}{X_t^i} \\
		&= \underbrace{\Big(\Kop{\mu_t^{i-1}}{X_t^i}-\Kop{\bar\nu_t^{i-1}}{X_t^i}\Big)}_{=:A_t^i}
		+\underbrace{\Big(\Kop{\bar\nu_t^{i-1}}{X_t^i}-\Kop{\bar\rho_t}{X_t^i}\Big)}_{=:B_t^i}.
	\end{aligned}
\end{equation}
Here $A_t^i$ captures the discrepancy between the empirical conditional measure and its conditional counterpart,
whereas $B_t^i$ compares the average   conditional  measure to the limit measure.

The term $A_t^i$ constitutes the main new difficulty: it compares an empirical input with its conditional mean, both evaluated at the random (and correlated) point $X_t^i$. We control $A_t^i$ in two steps: (i) we establish a sub-Gaussian estimate under the decoupled law $Q^{1:i}_{[0,t]}$, in which the last coordinate is independent of the preceding ones; and (ii) we transfer this estimate back to the interacting law via the Donsker--Varadhan variational formula for instance as in \cite{jabin2018quantitative}. The remaining term $B_t^i$ is handled using Pinsker's inequality.

\medskip

For each $j\ge2$ and each $t\in[0,T]$, we want to work with the conditional law of $X_t^j$ given the predecessor \emph{time-$t$} configuration
$X_t^{1:j-1}:=(X_t^1,\dots,X_t^{j-1})$. Since we will integrate in time, we need a jointly measurable version in $(t,\omega)$.
This is provided by disintegrating the time-space measure $dt\otimes\Law(X_t^{1:j})$.

\begin{lemma}\label{lem:time-disintegration}
	For each $j\ge2$ there exist probability kernels
	\[
	f_t^j:\ (\R^d)^{j-1}\to\mathcal P(\R^d),
	\qquad (t,x^{1:j-1})\mapsto f_t^j(x^{1:j-1},\cdot),
	\]
	which are jointly Borel measurable in $(t,x^{1:j-1})$ and satisfy
	\begin{equation}\label{eq:time-disintegration}
		\ud t\otimes \Law(X_t^{1:j})(\ud x^{1:j})
		= \ud t\otimes \Law(X_t^{1:j-1})(\ud x^{1:j-1})\, f_t^j(x^{1:j-1}, \ud x_j).
	\end{equation}
	In particular,
	\begin{equation}\label{eq:nu-def-slice}
		\nu_t^j(\omega):=f_t^j(X_t^{1:j-1}(\omega),\cdot)
		\quad\text{satisfies}\quad
		\nu_t^j=\Law(X_t^j\mid X_t^{1:j-1})\ \ \text{in the } \ud t\times\mathbb{P}\text{-a.s.\ sense}.
	\end{equation}
\end{lemma}

\begin{proof}
	It suffices to apply  the disintegration theorem (e.g. \cite[Theorem~5.3.1]{ambrosio2008gradient}) to the finite measure
	$dt\otimes \Law(X_t^{1:j})(dx^{1:j})$ on $[0,T]\times(\R^d)^j$ with respect to the projection
	$(t,x^{1:j})\mapsto (t,x^{1:j-1})$.
	This yields a jointly measurable kernel $f_t^j(x^{1:j-1},dx_j)$ such that \eqref{eq:time-disintegration} holds.
	The conditional-law interpretation \eqref{eq:nu-def-slice} follows by testing \eqref{eq:time-disintegration} against bounded measurable functions.
\end{proof}

 We will use the   conditional measures
\[
\nu_t^1:=\Law(X_t^1),
\qquad
\nu_t^j:=\Law(X_t^j\mid X_t^{1:j-1})
       =\mathbb E(\delta_{X_t^j}\mid X_t^{1:j-1}),\quad j\ge2,
\] 
and the average conditional measures
\[
\bar\nu_t^{i-1}:=\frac{1}{i-1}\sum_{j<i}\nu_t^j.
\]
 
 Note that $(t,\omega)\mapsto \bar\nu_t^{i-1}(\omega)$ is jointly measurable thanks to Lemma~\ref{lem:time-disintegration}.

\medskip

We next state a sub-Gaussian estimate  for averages of bounded $\R^d$-valued martingale differences.
\begin{lemma}\label{lem:md-exp}
Let   \((\mathcal G_j)_{j=0}^n\)  be a filtration and let   \((D_j)_{j=1}^n\) be \(\R^d\) -valued random vectors such that   \(D_j\) is \(\mathcal G_j\) -measurable,   \(\E[D_j\mid\mathcal G_{j-1}]=0\)  a.s., and   \(|D_j|\le L\)  a.s.   Set \(S_n:=\frac1n\sum_{j=1}^nD_j\) . Then there exists a   constant \(c_d>0\), depending only on the dimension \(d\), such that
\[
\E\exp\!\left(c_d\,\frac{n}{L^2}|S_n|^2\right)\le 2.
\] 
In particular, one may take   \(c_d=1/(6d)\) .
\end{lemma}

\begin{proof}
  Write \(D_j=(D_j^1,\ldots,D_j^d)\) and \(S_n=(S_n^1,\ldots,S_n^d)\). For each coordinate \(a=1,\ldots,d\), the scalar sequence \((D_j^a)_{j=1}^n\) is a martingale-difference sequence with \(|D_j^a|\le L\). The scalar  Azuma--Hoeffding inequality  gives
  \[
\mathbb P(|S_n^a|\ge r)\le 2\exp\!\left(-\frac{nr^2}{2L^2}\right),
\qquad r\ge0.
\] 
  As in the standard tail-to-moment computation, this implies
\[
\E\exp\!\left(\frac{n}{6L^2}|S_n^a|^2\right)\le 2,
\qquad a=1,\ldots,d.
\] 
  Since \(|S_n|^2=\sum_{a=1}^d |S_n^a|^2\), Hölder's inequality yields
\[
\E\exp\!\left(\frac{n}{6dL^2}|S_n|^2\right)
=
\E\prod_{a=1}^d
\exp\!\left(\frac{n}{6dL^2}|S_n^a|^2\right)
\le
\prod_{a=1}^d
\left[
\E\exp\!\left(\frac{n}{6L^2}|S_n^a|^2\right)
\right]^{1/d}
\le 2.
\]
 This proves the claim .
\end{proof}

Fix   \(i\ge2\) and a time \(t\) for which the conditional kernels below are defined , and recall the decoupled reference law
\[
Q^{1:i}_{[0,t]}:=P^{1:i-1}_{[0,t]}\otimes \bar P_{[0,t]}.
\]
Under the reference law $Q^{1:i}_{[0,t]}$, the last coordinate is independent of the predecessors.  
Lemma~\ref{lem:md-exp} has a direct consequence for our setting, which we summarize in the following lemma.

\begin{lemma}\label{lem:A-exp-Q}
	Fix $i\ge2$  . For Lebesgue-a.e.\ $t\in[0,T]$, let  $Q^{1:i}_{[0,t]}:=P^{1:i-1}_{[0,t]}\otimes \bar P_{[0,t]}$, denote by $Z^i$
	the $i$-th coordinate under $Q^{1:i}_{[0,t]}$, such that  $\Law(Z^i) = \bar P_{[0,t]}$ and $Z^i$ is independent of $(X^1,\dots,X^{i-1})$. 
	Then there exists a   constant \(c_d>0\), depending only on \(d\), such that
	\[
	\E_{Q^{1:i}_{[0,t]}}
	\exp\!\left(\frac{c_d(i-1)}{\,\|K\|_\infty^2}\,
	\Big|\Kop{\mu_t^{i-1}}{Z_t^i}-\Kop{\bar\nu_t^{i-1}}{Z_t^i}\Big|^2\right)\le 2 .
	\] 
\end{lemma}

\begin{proof}
	All expectations in this proof are taken  under the law  $Q^{1:i}_{[0,t]}$.
	Fix   such a time \(t\) and fix \(x\in\R^d\).

We use the convention \(\mathcal G_t^0=\{\varnothing,\Omega\}\) and
\(\nu_t^1=\Law(X_t^1)\).

For \(1\le j<i\) , set the time-  \(t\)  predecessor sigma-fields
	  \[
	\mathcal G_t^j:=\sigma(X_t^1,\dots,X_t^j),
	\]
	and define
	\[
	D_j(x):=K(x,X_t^j)-\mathbb E[K(x,X_t^j)\mid \mathcal G_t^{j-1}]\in\R^d,
	\qquad
	S_{i-1} (x):=\frac1{i-1}\sum_{j<i}D_j(x).
	\] 
	Then $D_j(x)$ is $\mathcal G_t^{j}$-measurable, $\E[D_j(x)\mid\mathcal G_t^{j-1}]=0$,
	and $|D_j(x)|\le 2\|K\|_\infty$.
	Moreover, by definition of $\nu_t^j$, one has
	\[
	\int_{\R^d}K(x,y)\,\nu_t^j(dy)
	= \E\!\left[K(x,X_t^j)\mid \mathcal G_t^{j-1}\right]
	\quad\text{in the } \ud t\times\mathbb{P}\text{-a.s.\ sense}. 
	\]
	Hence 
	\[
	\Kop{\mu_t^{i-1}}{x}-\Kop{\bar\nu_t^{i-1}}{x}
	= \frac1{i-1}\sum_{j<i}\Big(K(x,X_t^j)-\int_{\R^d } K(x,y)\nu_t^j(dy)\Big)
	= S_{i-1}(x).
	\]

		We apply  Lemma~\ref{lem:md-exp} to the martingale differences $(D_j(x))_{1\le j<i}$
		with $n=i-1$ and $L=2\|K\|_\infty$. After decreasing the constant \(c_d\) by a universal factor, if necessary, we obtain for any fixed  $x \in \mathbb{R}^d$,  
	  \[
	\E\exp\!\left(\frac{c_d(i-1)}{\,\|K\|_\infty^2}\,|S_{i-1} (x)|^2\right)\le 2,
	\] 
	where the constant   $c_d>0$ depends only on \(d\) . 
	Now return to the random evaluation point $x=Z_t^i$.
	Under the law $Q^{1:i}_{[0,t]}$, the $i$-th coordinate $Z_t^i$ is independent of the predecessors $(X_t^1,\dots,X_t^{i-1})$,
	so conditioning on $Z_t^i$ does not change the law of the predecessors and the above bound remains valid:
	  \[
	\E\exp\!\left(\frac{c_d(i-1)}{\,\|K\|_\infty^2}\,|S_{i-1}(Z_t^i)|^2\right)
	= \E\Big[\ \E\Big[\exp\!\Big(\frac{c_d(i-1)}{\,\|K\|_\infty^2}\,|S_{i-1}(Z_t^i)|^2\Big)\,\Big|\,Z_t^i\Big]\ \Big]
	\le 2.
	\] 
	Since $S_{i-1}(Z_t^i)=\Kop{\mu_t^{i-1}}{Z_t^i}-\Kop{\bar\nu_t^{i-1}}{Z_t^i}$, the result then follows.
\end{proof}

\begin{remark}
	We now tailor the previous lemmas---in particular Lemma~\ref{lem:A-exp-Q}---to our purpose of estimating the term $A_t^i$.
	However, the main message can be stated more generally as follows.
	Let $F_N$ be a probability density on $(\R^d)^N$, and write
	$F_N=\Law(X_1,\ldots,X_N)$.
	On the one hand, projecting to $\R^d$ yields the (random) empirical measure $\mu_N:=\frac1N\sum_{j=1}^N\delta_{X_j}$. On the other hand, one may disintegrate $F_N$ sequentially and define the conditional laws
	$
	\nu^j:=\Law\bigl(X_j\,\big|\,X_1,\ldots,X_{j-1}\bigr) 
	$
	as well as their averaged version
	$
	\nu_N:=\frac1N\sum_{j=1}^N \nu^j.
	$
	The preceding lemmas show that $\mu_N$ and $\nu_N$ are, in a precise sense, automatically close---and with essentially optimal control---without any additional structure beyond the existence of these conditional laws.
	  As we have already used above, this viewpoint extends verbatim to path-space distributions  .
 \end{remark}

We now transfer the estimate back to the true law $P^{1:i}_{[0,t]}$.
\begin{corollary}\label{cor:A-transfer-entropy}
Fix   \(i\ge2\). For Lebesgue-a.e.\ \(t\in[0,T]\), there exists a constant \(C_d>0\), depending only on \(d\), such that
\[
  \E\big|A_t^i\big|^2
  \le
  \frac{C_d\|K\|_\infty^2}{i-1}\big(R_i(t)+1\big).
\] 
\end{corollary}

\begin{proof}
Let $P:=P^{1:i}_{[0,t]}$ and $Q:=Q^{1:i}_{[0,t]}=P^{1:i-1}_{[0,t]}\otimes \bar P_{[0,t]}$.
Define $h(z):=\big|\Kop{\mu_t^{i-1}}{z_t^i}-\Kop{\bar\nu_t^{i-1}}{z_t^i}\big|^2$ on $\mathcal C_t^i$,
so that $\E_P[h]=\E|A_t^i|^2$.
By the Donsker--Varadhan variational formula, for any $\lambda>0$,
\[
  \E_P[h]\le \frac{1}{\lambda}\Ent(P\mid Q)+\frac{1}{\lambda}\log\E_Q[e^{\lambda h}].
\]
Here $\Ent(P\mid Q)=R_i(t)$ by  the definition  of the incremental relative entropy.
Take   $\lambda=\frac{c_d(i-1)}{\|K\|_\infty^2}$  and apply Lemma~\ref{lem:A-exp-Q} (with   \(c_d>0\) )
to get $\E_Q[e^{\lambda h}]\le 2$. Hence 
  \[
  \E|A_t^i|^2
  \le \frac{\|K\|_\infty^2}{c_d(i-1)}\,R_i(t)+\frac{\|K\|_\infty^2\log 2}{c_d(i-1)}
  \le \frac{C_d\|K\|_\infty^2}{i-1}\,(R_i(t)+1).
\] 
\end{proof}

We now turn to the estimate of the second term $B_t^i$ in \eqref{eq:Delta-decomp}.
At this point it is worth emphasizing the role of the auxiliary random measures $(\nu_t^j)_{j\ge2}$ obtained by
time--space disintegration (Lemma~\ref{lem:time-disintegration}).
By construction, $\nu_t^j$ is a version of the conditional marginal law of $X_t^j$ given the predecessor configuration
$X_t^{1:j-1}$, and $\bar\nu_t^{i-1}=\frac1{i-1}\sum_{j<i}\nu_t^j$ is the corresponding averaged conditional law.
Thus $\bar\nu_t^{i-1}$ should be viewed as the ``predictable'' mean of the empirical measure $\mu_t^{i-1}$ at time $t$,
and the pair $(\mu_t^{i-1},\bar\nu_t^{i-1})$ forms a canonical empirical/conditional-mean couple.

Such conditional-mean random measures appear naturally in the weak-convergence approach to large deviations and controlled
limits (see, e.g., \cite{dupuis2011ldp}), and have also been used in our earlier work \cite{wang2022mean}.
A key advantage is that they inherit regularity from the underlying law in a robust way:
information--theoretic quantities such as relative entropy (and, in diffusive settings, Fisher information)
propagate along these conditional laws.
In the present bounded-kernel setting we only need a very soft consequence of this principle:
$\bar\nu_t^{i-1}$ is much closer to the limit marginal $\bar\rho_t$ than the raw empirical measure $\mu_t^{i-1}$.
For instance, the   random  part of $\mu_t^{i-1}-\bar\rho_t$ does not vanish in total variation. 
This replacement is  what makes the $B$--term estimate possible; without passing to the disintegrated conditional
measures, one cannot close the argument at the level of $\|\cdot\|_{\TV}$.

 \begin{lemma}\label{lem:Bterm}
For every   \(i\ge2\) and for Lebesgue-a.e.\ \(t\in[0,T]\) ,
\[
  \E\Big|\Kop{\bar\nu_t^{i-1}}{X_t^i}-\Kop{\bar\rho_t}{X_t^i}\Big|^2
  \le \frac{2\|K\|_\infty^2}{(i-1)^2}\Big(\sum_{j<i}\sqrt{R_j(t)}\Big)^2 .
\]
\end{lemma}

\begin{proof}
Fix $x\in\R^d$ and $\eta\in\mathcal P(\R^d)$. Then
\[
  \big|\Kop{\eta}{x}-\Kop{\bar\rho_t}{x}\big|
  =\Big|\int_{\R^d}K(x,y)\,(\eta-\bar\rho_t)(\ud y)\Big|
  \le \|K\|_\infty\,\|\eta-\bar\rho_t\|_{\TV}.
\]
Since $\bar\nu_t^{i-1}=\frac{1}{i-1}\sum_{j<i}\nu_t^j$, by the triangle inequality one obtains that 
\[
  \|\bar\nu_t^{i-1}-\bar\rho_t\|_{\TV}
  \le \frac{1}{i-1}\sum_{j<i}\|\nu_t^j-\bar\rho_t\|_{\TV}
  \le \frac{1}{i-1}\sum_{j<i}\sqrt{2\,\Ent(\nu_t^j\mid\bar\rho_t)},
\]
where the last inequality is simply  Pinsker's inequality (see Lemma \ref{lem:entropy-tools} (i)). 

Therefore, for every $x$,
\[
  \big|\Kop{\bar\nu_t^{i-1}}{x}-\Kop{\bar\rho_t}{x}\big|
  \le \frac{\sqrt2\,\|K\|_\infty}{i-1}\sum_{j<i}\sqrt{\Ent(\nu_t^j\mid\bar\rho_t)}.
\]
Evaluate at $x=X_t^i$, square, and set $H_j:=\Ent(\nu_t^j\mid\bar\rho_t)$ to get that 
\[
  \Big|\Kop{\bar\nu_t^{i-1}}{X_t^i}-\Kop{\bar\rho_t}{X_t^i}\Big|^2
  \le \frac{2\|K\|_\infty^2}{(i-1)^2}\Big(\sum_{j<i}\sqrt{H_j}\Big)^2.
\]
Taking expectation and using Minkowski's inequality in $L^2$ for the nonnegative variables $\sqrt{H_j}$ yields that 
\[
  \E\Big(\sum_{j<i}\sqrt{H_j}\Big)^2
  \le \Big(\sum_{j<i}\sqrt{\E H_j}\Big)^2.
\]
Finally, by the data processing inequality (see Lemma \ref{lem:entropy-tools} (ii)) applied to the projection
$(X^{1:j}_{[0,t]})\mapsto (X^{1:j}_{t})$, we have
\[
  \E H_j=\E\big[\Ent(\nu_t^j\mid\bar\rho_t)\big]\le R_j(t),
\]
where for \(j=1\) this follows directly from data processing applied to the first-particle path entropy \(R_1(t)\), and for \(j\ge2\) it follows from the disintegrated form \eqref{eq:R_i_conditional}.
This yields the claim.
\end{proof}

\subsection{Upper-envelope closure and tail chaos}\label{subsec:envelope}

Fix $i\ge2$ and $t\in[0,T]$.
By Lemma~  \ref{lem:Ri-girsanov} ,
\[
  R_i(t)
  = R_i(0)+\frac12\,\E\int_0^t\big|\sigma^{-1}\Delta_s^i\big|^2\ud s,
  \qquad
  \Delta_s^i:=\Kop{\mu_s^{i-1}}{X_s^i}-\Kop{\bar\rho_s}{X_s^i}.
\]
With the decomposition $\Delta_s^i=A_s^i+B_s^i$ from Subsection~\ref{subsec:replacement},
\[
  A_s^i:=\Kop{\mu_s^{i-1}}{X_s^i}-\Kop{\bar\nu_s^{i-1}}{X_s^i},
  \qquad
  B_s^i:=\Kop{\bar\nu_s^{i-1}}{X_s^i}-\Kop{\bar\rho_s}{X_s^i},
\]
we obtain the following integral inequality.

 \begin{proposition}\label{prop:Ri-integral-ineq}
There exists a   constant \(C_d>0\), depending only on \(d\),  such that for all $i\ge2$ and all $t\in[0,T]$,
  \begin{equation}\label{eq:Ri-integral-ineq}
  R_i(t)
  \le R_i(0)
  +\kappa\int_0^t\Bigg(
      \frac{C_d}{i-1}\,R_i(s)
      +\frac{C_d}{i-1}
      +\frac{C_d}{(i-1)^2}\Big(\sum_{j<i}\sqrt{R_j(s)}\Big)^2
  \Bigg)\ud s,
\end{equation} 
where $\kappa:=\|\sigma^{-1}\|^2\|K\|_\infty^2$.
\end{proposition}

\begin{proof}
Since $|\Delta_s^i|^2\le 2|A_s^i|^2+2|B_s^i|^2$,
\[
  R_i(t)-R_i(0)
  =\frac12\,\E\int_0^t|\sigma^{-1}\Delta_s^i|^2\,ds
  \le \|\sigma^{-1}\|^2\int_0^t\Big(\E|A_s^i|^2+\E|B_s^i|^2\Big) \ud s.
\]
We estimate $\E|A_s^i|^2$ by Corollary~\ref{cor:A-transfer-entropy} and $\E|B_s^i|^2$ by Lemma~\ref{lem:Bterm}.
Absorbing numerical constants into   \(C_d\) yields  \eqref{eq:Ri-integral-ineq}.
\end{proof}

The last term in \eqref{eq:Ri-integral-ineq} involves the increments up to $i-1$. 
We now close \eqref{eq:Ri-integral-ineq} by controlling the coupling term
$\big(\sum_{j<i}\sqrt{R_j}\big)^2$ through an upper envelope along the index.

\begin{lemma}\label{lem:sqrtR-sum-envelope-i}
Set \(M_1(t):=R_1(t)\). For each \(i\ge2\), define the \(i\)-th upper envelope \(M_i(t)\) as
\[
  M_i(t):=\max\Big\{R_1(t),\ \max_{2\le k\le i}(k-1)\,R_k(t)\Big\},\qquad t\in[0,T].
\]
Then for every $i\ge2$ and every $s\in[0,T]$,
\begin{equation}\label{eq:sqrtR-sum-envelope-i}
  \Big(\sum_{j<i}\sqrt{R_j(s)}\Big)^2 \le 10\,(i-1)\,M_{i-1}(s).
\end{equation}
\end{lemma}

\begin{proof}
By definition of $M_{i-1}(t)$, one has that  $R_1(s)\le M_{i-1}(s)$, and   $(j-1)R_j(s)\le M_{i-1}(s)$ for  $2\le j\le i-1$ .  Therefore 
\[
  \sqrt{R_1(s)}\le \sqrt{M_{i-1}(s)},\qquad
  \sqrt{R_j(s)}\le \frac{\sqrt{M_{i-1}(s)}}{\sqrt{j-1}}\quad(2\le j\le i-1).
\]
Therefore,
\[
  \sum_{j<i}\sqrt{R_j(s)}
  \le \sqrt{M_{i-1}(s)}+\sqrt{M_{i-1}(s)}\sum_{j=2}^{i-1}\frac1{\sqrt {j-1}}
  \le \sqrt{M_{i-1}(s)}+2\sqrt{(i-1)M_{i-1}(s)},
\]
where we used $\sum_{m=1}^{n}m^{-1/2}\le 2\sqrt n$ and $i-2\le i-1$.
Using $(a+b)^2\le 2a^2+2b^2$ yields
\[
  \Big(\sum_{j<i}\sqrt{R_j(s)}\Big)^2
  \le 2M_{i-1}(s)+8(i-1)M_{i-1}(s)
  \le 10(i-1)M_{i-1}(s),
\]
which proves \eqref{eq:sqrtR-sum-envelope-i}.
\end{proof}

We thus combine Proposition \ref{prop:Ri-integral-ineq} and the above lemma \ref{lem:sqrtR-sum-envelope-i} to obtain the following estimates for the   envelope  $M_i$. 

\begin{proposition}\label{prop:envelope-i}
There exists a   constant \(C_d>0\), depending only on \(d\),  such that for every $i\ge2$ and every $t\in[0,T]$,
  \begin{equation}\label{eq:Mi-closed}
  M_i(t)\le M_i(0)+\kappa\int_0^t\big(C_d\,M_i(s)+C_d\big) \ud s.
\end{equation} 
Consequently, there exists $C_T>0$ depending only on   \(d\) and \((T,\kappa)\)  such that for every $i\ge2$ and every $t\in[0,T]$,
\begin{equation}\label{eq:Mi-bound}
  M_i(t)\le C_T\bigl(1+M_i(0)\bigr),
  \qquad\text{and hence}\qquad
  R_i(t)\le \frac{C_T}{i-1}\bigl(1+M_i(0)\bigr).
\end{equation}
\end{proposition}

\begin{proof}
Fix $i\ge2$ and $t\in[0,T]$. 
Using Lemma~\ref{lem:sqrtR-sum-envelope-i} and $M_{i-1}(s)\le M_i(s)$ gives
\[
  \E|B_s^i|^2
  \le \frac{C\|K\|_\infty^2}{(i-1)^2}\cdot 10(i-1)M_{i-1}(s)
  \le \frac{C\|K\|_\infty^2}{i-1}\,M_i(s).
\]
Plugging the  above estimate into    Proposition  \ref{prop:Ri-integral-ineq}  and recalling $\kappa=\|\sigma^{-1}\|^2\|K\|_\infty^2$ yields
\[
  R_i(t)\le R_i(0)+\kappa\int_0^t\Big(\frac{C}{i-1}\,R_i(s)+\frac{C}{i-1}+\frac{C}{i-1}\,M_i(s)\Big)\ud s.
\]
We then multiply by $(i-1)$ and use $(i-1)R_i(s)\le M_i(s)$ (for $i\ge2$) to get 
\[
  (i-1)R_i(t)\le (i-1)R_i(0)+\kappa\int_0^t\big(C\,M_i(s)+C\big) \ud s.
\]
Taking the maximum over $2\le k\le i$ gives the part of \(M_i(t)\) involving the scaled increments. The remaining term \(R_1(t)\) is controlled by the first-particle estimate in Lemma~\ref{lem:Ri-girsanov}: \(R_1(t)\le R_1(0)+\kappa t/2\le M_i(0)+\kappa\int_0^t C\,\mathrm ds\), and is absorbed into the same right-hand side after enlarging the constant. Recalling the definition of $M_i$ then implies \eqref{eq:Mi-closed}. The result follows by
Gronwall's inequality.
\end{proof}

With the previous preparations in place, we can now prove our first main result, Theorem~\ref{thm:incremental}, directly.

\begin{proof}[Proof of Theorem~\ref{thm:incremental}]
Fix \(i\ge2\), and let \(M_i(t)\) be the envelope defined in Lemma~\ref{lem:sqrtR-sum-envelope-i}. Then \(M_i(0)=M_{0,i}\), and the desired estimate follows from \eqref{eq:Mi-bound}.
\end{proof}

We next derive the   propagation  of chaos for tail particles  directly from the incremental decomposition.

\begin{proof}[Proof of Corollary~\ref{cor:tail-poc}]
 By data processing, projecting \(P^{1:N}_{[0,T]}\)  to the last   \(m\) coordinates gives
\[
\Ent\!\left(
\Law(Y^{N,m})
\,\middle|\,
\bar P_{[0,T]}^{\otimes m}
\right)
\le
\Ent\!\left(
P^{1:N}_{[0,T]}
\,\middle|\,
P^{1:N-m}_{[0,T]}\otimes \bar P_{[0,T]}^{\otimes m}
\right).
\] 
  The  chain rule along the particle index   gives
\[
\Ent\!\left(
P^{1:N}_{[0,T]}
\,\middle|\,
P^{1:N-m}_{[0,T]}\otimes \bar P_{[0,T]}^{\otimes m}
\right)
=
\sum_{k=N-m+1}^{N}R_k(T).
\] 
  
If \(m\le N/2\), then every \(k\in\{N-m+1,\ldots,N\}\) satisfies
\[
k-1\ge N-m\ge N/2.
\]

Using Theorem~\ref{thm:incremental} and the  i.i.d.\   initial condition, we obtain

\[
\sum_{k=N-m+1}^{N} R_k(T)
\le C_T\sum_{k=N-m+1}^{N}\frac1{k-1}
\le C_T\frac{m}{N}.
\] 
  
 \end{proof}

\subsection{Global path-space entropy via Hardy's inequality}\label{subsec:global-entropy}

In this subsection we present the detailed proof of Corollary~\ref{cor:global-entropy}, namely the global path-space entropy estimate.
We emphasize that it does \emph{not} follow from a naive summation of the pointwise bound $R_i(T)\lesssim (i-1)^{-1}$.
Instead, we sum the integral inequality \eqref{eq:Ri-integral-ineq} and control the coupled term
\[
\sum_{i=2}^N (i-1)^{-2}\Big(\sum_{j<i}\sqrt{R_j}\Big)^2
\]
via a discrete Hardy inequality recalled below; see \cite[Theorem~326]{hardy1952inequalities} for a proof.

  \begin{lemma}[Discrete Hardy inequality for the case $p=2$]
	 For any nonnegative sequence $(a_k)_{k\ge1}$ and any $n\in\mathbb N$, one has that 
	\begin{equation}
		\label{lem:hardy}
		\sum_{k=1}^n \frac{1}{k^2}\Big(\sum_{j=1}^k a_j\Big)^2
		\le 4\sum_{k=1}^n a_k^2.
	\end{equation}
\end{lemma}

We now proceed to prove Corollary~\ref{cor:global-entropy}. 

\begin{proof}[Proof of Corollary~\ref{cor:global-entropy}]
We recall that $S_N(t):=\Ent(P^{1:N}_{[0,t]}\mid \bar P_{[0,t]}^{\otimes N})$ and also  by the chain rule (see (i) in  Corollary \ref{cor:incremental} ),  
\[
S_N(t)=\sum_{i=1}^N R_i(t), 
\]
for all $ t\in[0,T]$.

We first sum \eqref{eq:Ri-integral-ineq} over $i=2,\dots,N$:
for all $t\in[0,T]$,
\begin{align*}
  \sum_{i=2}^N R_i(t)
  &\le \sum_{i=2}^N R_i(0)
  +\kappa\int_0^t\Bigg(
      C\sum_{i=2}^N\frac{1}{i-1}R_i(s)
      +C\sum_{i=2}^N\frac{1}{i-1}
\\
  &\hspace{6em}
      +C\sum_{i=2}^N\frac{1}{(i-1)^2}\Big(\sum_{j<i}\sqrt{R_j(s)}\Big)^2
  \Bigg)\ud s,
\end{align*}
for a   constant \(C_d>0\) depending only on \(d\) .

We now incorporate the $i=1$ term explicitly:
by \eqref{eq:R1-explicit},
\[
  R_1(t)\le R_1(0)+\frac12\,\kappa\,t \le R_1(0)+\frac12\,\kappa\,T,
  \qquad t\in[0,T].
\]
Adding this to the previous inequality gives, for all $t\in[0,T]$,
\begin{equation}\label{eq:SN-ineq}
  S_N(t)
  \le S_N(0)+\frac12\,\kappa\,T
  +\kappa\int_0^t\Bigg(
      C\sum_{i=2}^N\frac{1}{i-1}R_i(s)
      +C\sum_{i=2}^N\frac{1}{i-1}
      +C\sum_{i=2}^N\frac{1}{(i-1)^2}\Big(\sum_{j<i}\sqrt{R_j(s)}\Big)^2
  \Bigg) \ud s.
\end{equation}

We bound the three sums in \eqref{eq:SN-ineq} as follows.
First,
\[
  \sum_{i=2}^N\frac{1}{i-1}R_i(s)\le \sum_{i=2}^N R_i(s)\le S_N(s).
\]
Second,
\[
  \sum_{i=2}^N\frac{1}{i-1}\le 1+\log N.
\]
Third, we apply Hardy's inequality \eqref{lem:hardy} with $a_j=\sqrt{R_j(s)}$ and $k=i-1$:
\[
  \sum_{i=2}^N\frac{1}{(i-1)^2}\Big(\sum_{j<i}\sqrt{R_j(s)}\Big)^2
  =\sum_{k=1}^{N-1}\frac{1}{k^2}\Big(\sum_{j=1}^k \sqrt{R_j(s)}\Big)^2
  \le 4\sum_{j=1}^{N-1} R_j(s)
  \le 4S_N(s).
\]
Inserting these bounds into \eqref{eq:SN-ineq} yields
\[
  S_N(t)\le S_N(0)+\frac12\,\kappa\,T
  +\kappa\int_0^t\big(C\,S_N(s)+C(1+\log N)\big)\ud s.
\]
Gronwall's inequality implies
\[
  S_N(T)\le C_T\big(S_N(0)+\log N\big),
\]
for a constant $C_T>0$ depending only on  \(d\) and  $(T,\kappa)$.
\end{proof}

  \subsection{Discussion: general sequential weights and effective sample size}
\label{subsec:general-weights}

The canonical system  \eqref{eq:seqSDE}   uses the uniform predecessor empirical measure, corresponding to the step size \(\alpha_i=1/i\). At the level of the entropy recursion, the same mechanism suggests the following effective-sample-size picture for more general sequential weights. We state this scale here to connect the present incremental entropy method with the weighting schemes of Du--Jiang--Li~\cite{du2023sequential}. Let \((\alpha_i)_{i\ge1}\subset(0,1]\) be deterministic step sizes with \(\alpha_1=1\), and define recursively
\[
\mu_t^i=(1-\alpha_i)\mu_t^{i-1}+\alpha_i\delta_{X_t^i},
\qquad i\ge2,
\]
with \(\mu_t^1=\delta_{X_t^1}\). Equivalently,
\[
\mu_t^i=\sum_{k=1}^i w_{i,k}\delta_{X_t^k},
\qquad
w_{i,i}=\alpha_i,\qquad
w_{i,k}=\alpha_k\prod_{\ell=k+1}^i(1-\alpha_\ell)\quad(k<i),
\]
with \(\sum_{k=1}^iw_{i,k}=1\). The relevant variance scale is the squared weight mass
\[
\theta_i:=\sum_{k=1}^iw_{i,k}^2,
\qquad
N_{\rm eff}(i):=\theta_i^{-1}.
\]
For the canonical choice \(\alpha_i=1/i\), one has \(w_{i,k}=1/i\), hence \(\theta_i=1/i\) and \(N_{\rm eff}(i)=i\).

The weighted analogue of  the   martingale-difference estimate in Lemma~\ref{lem:md-exp} suggests that the empirical--conditional fluctuation term has variance scale \(\theta_{i-1}\). Thus one expects the entropy scale
\[
R_i(T)\lesssim \theta_{i-1}=N_{\rm eff}(i-1)^{-1}.
\]
together with a compatible initial entropy assumption, for instance   i.i.d.\ initial data with law \(\bar\rho_0\).
Carrying this out would mainly add predictable-bias bookkeeping for the \(B\)-term. Since this extension is not used below and would obscure the canonical uniform case, we leave the general-weight statement out of the present paper.

  For power-law step sizes the behavior of \(\theta_i\) depends on the precise regime. If
\(\alpha_i\asymp i^{-r}\) with \(0<r<1\), then \(\theta_i\asymp i^{-r}\), and hence
\(N_{\rm eff}(i)\asymp i^r\). If \(\alpha_i=c/i\), then the critical behavior depends on \(c\):
\[
\theta_i\asymp
\begin{cases}
i^{-1}, & c>1/2,\\
(\log i)/i, & c=1/2,\\
i^{-2c}, & 0<c<1/2.
\end{cases}
\]
Thus the corresponding entropy scale is \(R_i(T)\lesssim\theta_{i-1}\). If \(r>1\), then \(\sum_i\alpha_i<\infty\), the early weights do not vanish, and \(\theta_i\) need not converge to zero; convergence to the deterministic mean-field limit cannot be expected in general without additional contractive structure.
\paragraph{Possible extension to systems with singular interactions.}
We briefly comment on singular kernels. For mildly singular interactions, one can expect the present incremental strategy to combine with the  marginal-entropy/Fisher-information   approach developed for classical systems with singular forces; see, for example, \cite{jabin2018quantitative,serfaty2020mean}. In such a framework, one would work with time-marginal entropies rather than path-space laws and use integration by parts to compensate the singularity, while keeping track of the directed triangular structure of \eqref{eq:seqSDE}. This should also be compatible with  uniform-in-time   variants on compact state spaces, in the spirit of  \cite{guillin2024uniform}.

  We do not pursue  this generality here.   For strongly singular kernels, such as Coulomb-type interactions, it remains less clear how  modulated  energy or modulated  free energy   methods \cite{serfaty2020mean,bresch2020mean} should be adapted to the sequential architecture. This  appears to be an interesting direction for future work.

\section{Convergence of empirical measures in negative Sobolev norms}\label{sec:consequences}

In this section we prove Theorem~\ref{thm:empirical}. Throughout the section the initial variables \((X_0^i)_{i\ge1}\) are i.i.d.\ with common law \(\bar\rho_0\). We write the sequential dynamics as the sum of a common McKean--Vlasov velocity and a triangular interaction error. Namely, set
\[
V_t(x):=b(t,x)+\Kop{\bar\rho_t}{x}.
\]

We also use the conventions
\[
\mu_t^0:=0,\qquad
\mu_t^i:=\frac1i\sum_{j=1}^i\delta_{X_t^j}\quad(i\ge1),
\]
and define, for \(i\ge1\),
\[
\Delta_t^i:=\Kop{\mu_t^{i-1}}{X_t^i}-\Kop{\bar\rho_t}{X_t^i}.
\]
Thus \(\Delta_t^1=-\Kop{\bar\rho_t}{X_t^1}\), and the particle dynamics can be written uniformly as
\begin{equation}\label{eq:Xi-v-plus-Delta}
  \mathrm dX_t^i
  =
  \big(V_t(X_t^i)+\Delta_t^i\big)\,\mathrm dt
  +\sigma\,\mathrm dB_t^i,
  \qquad i=1,\dots,N.
\end{equation}

Fix a Sobolev index
\begin{equation}\label{eq:beta-choice}
  \beta>\frac d2+2.
\end{equation}

With the convention introduced in Section~\ref{sec:prelim}, the Bessel kernel \(G_\beta=(1-\Delta)^{-\beta}\delta_0\) belongs to \(C_b^2(\mathbb R^d)\), with bounded first and second derivatives.

We consider the rescaled empirical error and its Bessel potential
\[
  \eta_t^N:=\sqrt{N}(\mu_t^N-\bar\rho_t),
  \qquad
  \mu_t^N:=\frac1N\sum_{i=1}^N\delta_{X_t^i},
  \qquad
  \Phi_t^N:=G_\beta*\eta_t^N.
\]

The proof has two steps. First, we derive an energy identity for \(\|\eta_t^N\|_{H^{-\beta}}^2\) and turn it into an a priori estimate with a coercive \(H^{-\beta+1}\)-dissipation term. Second, we estimate the three quantities left on the right-hand side of this energy inequality: the initial empirical fluctuation, the accumulated triangular interaction error, and the martingale contribution.

\subsection{Kernel energy identity and a priori estimate}\label{subsec:energy}

We consider the kernel energy
\begin{equation}\label{eq:energy-def}
  \|\eta_t^N\|_{H^{-\beta}}^2
  =\iint_{\R^d\times\R^d} G_\beta(x-y)\,\eta_t^N(\mathrm dx)\,\eta_t^N(\mathrm dy)
  =\big\langle \eta_t^N,\Phi_t^N\big\rangle .
\end{equation}

\begin{lemma}\label{lem:energy-identity}
Let $\beta>\frac d2+2$. Then, for every $t\in[0,T]$,
  \begin{equation}\label{eq:energy-identity}
\begin{aligned}
 \|\eta_t^N\|_{H^{-\beta}}^2
  &=\|\eta_0^N\|_{H^{-\beta}}^2
    +c_{\sigma,\beta}t
    +\int_0^t 2\Big\langle \eta_r^N,\mathcal L \Phi_r^N\Big\rangle\,\mathrm dr
    +\int_0^t 2\Big\langle \eta_r^N, V_r\cdot\nabla \Phi_r^N\Big\rangle\,\mathrm dr \\
  &\quad
    +\int_0^t \frac{2}{\sqrt{N}}\sum_{i=1}^N \nabla\Phi_r^N(X_r^i)\cdot \Delta_r^i\,\mathrm dr
    +\mathfrak M_t^N,
\end{aligned}
\end{equation} 
where
\[
\mathcal L\varphi:=\frac12\text{Tr}(\sigma\sigma^\top\nabla^2\varphi),
\qquad
c_{\sigma,\beta}:=-\operatorname{Tr}\!\big(\sigma\sigma^\top\nabla^2G_\beta(0)\big)\ge0,
\]
Moreover,
\begin{equation}\label{eq:energy-martingale}
  \mathfrak M_t^N
  :=\frac{2}{\sqrt{N}}\sum_{i=1}^N\int_0^t \nabla\Phi_r^N(X_r^i)\cdot \sigma\,\mathrm d B_r^i
\end{equation} 
is a real-valued continuous martingale.
\end{lemma}
\begin{proof}
Expanding $\eta_t^N=\sqrt{N}(\mu_t^N-\bar\rho_t)$ yields the decomposition
\[
   \frac{1}{N}\|\eta_t^N\|_{H^{-\beta}}^2=A_t^N-2B_t^N+C_t,
\]
where
\[
  A_t^N:=\iint G_\beta(x-y)\,\mu_t^N(\mathrm dx)\,\mu_t^N(\mathrm dy)
       =\frac1{N^2}\sum_{i,j=1}^N G_\beta(X_t^i-X_t^j),
\]
\[
  B_t^N:=\int (G_\beta*\bar\rho_t)(x)\,\mu_t^N(\mathrm dx)
       =\frac1N\sum_{i=1}^N \phi_t(X_t^i),
  \qquad \phi_t:=G_\beta*\bar\rho_t,
\]
\[
  C_t:=\iint G_\beta(x-y)\,\bar\rho_t(\mathrm dx)\,\bar\rho_t(\mathrm dy)
       =\langle \bar\rho_t,\phi_t\rangle.
\]
Since $G_\beta\in C_b^2(\R^d)$ and $\phi_t\in C_b^2(\R^d)$, It\^o's formula applies to $A_t^N$ and $B_t^N$.
For the time-dependent function \(\phi_t\), the derivative \(\partial_t\phi_t\) is understood by testing the weak Fokker--Planck equation against \(G_\beta(x-\cdot)\), which makes \(t\mapsto\phi_t(x)\) absolutely continuous for each \(x\).
We compute the differentials of $A_t^N$, $B_t^N$, and $C_t$ separately, and then sum them up.

\smallskip

\emph{It\^o formula for $A_t^N$.}
For $x=(x_1,\dots,x_N)$ set
\[
  A^N(x):=\frac1{N^2}\sum_{i,j=1}^N G_\beta(x_i-x_j),
  \qquad
  \mu_x^N:=\frac1N\sum_{i=1}^N\delta_{x_i}.
\]
Using the evenness of $G_\beta$ (hence $\nabla G_\beta$ is odd), a direct differentiation gives, for each $i$,
  \[
  \nabla_{x_i}A^N(x)=\frac{2}{N}\,\nabla\big(G_\beta*\mu_x^N\big)(x_i),
\] 
 and
\[
  \nabla_{x_i}^2A^N(x)
  =
  \frac{2}{N}\,\nabla^2\big(G_\beta*\mu_x^N\big)(x_i)
  -\frac{2}{N^2}\,\nabla^2G_\beta(0).
\]
Indeed, the diagonal term \(G_\beta(x_i-x_i)=G_\beta(0)\) is constant in \(x_i\), and therefore its second derivative along the \(x_i\)-variable is zero.
 Applying It\^o to $A_t^N=A^N(X_t^1,\dots,X_t^N)$ and using \eqref{eq:Xi-v-plus-Delta} yields
  \begin{equation}\label{eq:dA}
\begin{aligned}
  \mathrm d A_t^N
  &=\frac{2}{N}\sum_{i=1}^N \nabla(G_\beta*\mu_t^N)(X_t^i)\cdot \mathrm dX_t^i
    +\frac12\sum_{i=1}^N \text{Tr}\!\Big(\sigma\sigma^\top \nabla_{x_i}^2A^N(X_t)\Big)\,\mathrm dt \\
  &=2\Big\langle \mu_t^N,\, V_t\cdot\nabla(G_\beta*\mu_t^N)\Big\rangle\,\mathrm dt
    +\frac{2}{N}\sum_{i=1}^N \nabla(G_\beta*\mu_t^N)(X_t^i)\cdot\Delta_t^i\,\mathrm dt \\
  &\quad
    +2\Big\langle \mu_t^N,\,\mathcal L(G_\beta*\mu_t^N)\Big\rangle\,\mathrm dt
    -\frac1N\operatorname{Tr}\!\big(\sigma\sigma^\top\nabla^2G_\beta(0)\big)\,\mathrm dt \\
  &\quad
    +\frac{2}{N}\sum_{i=1}^N \nabla(G_\beta*\mu_t^N)(X_t^i)\cdot\sigma\,\mathrm dB_t^i .
\end{aligned}
\end{equation} 

\smallskip

\emph{It\^o formula for $B_t^N$.}
Recall $B_t^N=\langle \mu_t^N,\phi_t\rangle=\frac1N\sum_{i=1}^N\phi_t(X_t^i)$ with $\phi_t=G_\beta*\bar\rho_t$.
Applying It\^o to $\phi_t(X_t^i)$ and summing over $i$ gives
  \begin{equation}\label{eq:dB}
\begin{aligned}
  \mathrm d B_t^N
  &=\Big\langle \mu_t^N,\,\partial_t\phi_t+\mathcal L\phi_t+V_t\cdot\nabla\phi_t\Big\rangle\,\mathrm dt
    +\frac{1}{N}\sum_{i=1}^N \nabla\phi_t(X_t^i)\cdot\Delta_t^i\,\mathrm dt \\
  &\quad
    +\frac{1}{N}\sum_{i=1}^N \nabla\phi_t(X_t^i)\cdot\sigma\,\mathrm dB_t^i .
\end{aligned}
\end{equation} 

\smallskip

\emph{Time derivative of $C_t$.}
Since $C_t=\langle \bar\rho_t,\phi_t\rangle$ and $\phi_t=G_\beta*\bar\rho_t$, we have
\[
  \frac{\mathrm d}{\mathrm dt}C_t
  =\langle \partial_t\bar\rho_t,\phi_t\rangle+\langle \bar\rho_t,\partial_t\phi_t\rangle
  =\langle \partial_t\bar\rho_t,\phi_t\rangle+\langle \bar\rho_t,G_\beta*(\partial_t\bar\rho_t)\rangle
  =2\langle \partial_t\bar\rho_t,\phi_t\rangle.
\]
Therefore
\begin{equation}\label{eq:dC}
  \mathrm d C_t = 2\langle \partial_t\bar\rho_t,\phi_t\rangle\,\mathrm dt.
\end{equation}

\smallskip

\emph{Summation and cancellations.}
Combining \eqref{eq:dA}--\eqref{eq:dC} we obtain
  \begin{equation}\label{eq:dE-pre}
\begin{aligned}
  \mathrm d \|\eta_t^N\|_{H^{-\beta}}^2
  &=N\Big(\mathrm dA_t^N-2\,\mathrm dB_t^N+\mathrm dC_t\Big) \\
  &=2N\Big\langle \mu_t^N,\,\mathcal L(G_\beta*\mu_t^N)\Big\rangle\,\mathrm dt
    -2N\Big\langle \mu_t^N,\,\mathcal L\phi_t\Big\rangle\,\mathrm dt \\
  &\quad
    -\operatorname{Tr}\!\big(\sigma\sigma^\top\nabla^2G_\beta(0)\big)\,\mathrm dt \\
  &\quad
    +2N\Big\langle \mu_t^N,\, V_t\cdot\nabla(G_\beta*\mu_t^N)\Big\rangle\,\mathrm dt
    -2N\Big\langle \mu_t^N,\, V_t\cdot\nabla\phi_t\Big\rangle\,\mathrm dt \\
  &\quad
    +2\sum_{i=1}^N \Big(\nabla(G_\beta*\mu_t^N)-\nabla\phi_t\Big)(X_t^i)\cdot\Delta_t^i\,\mathrm dt \\
  &\quad
    -2N\Big\langle \mu_t^N,\,\partial_t\phi_t\Big\rangle\,\mathrm dt
    +2N\Big\langle \partial_t\bar\rho_t,\,\phi_t\Big\rangle\,\mathrm dt \\
  &\quad
    +2\sum_{i=1}^N \Big(\nabla(G_\beta*\mu_t^N)-\nabla\phi_t\Big)(X_t^i)\cdot\sigma\,\mathrm dB_t^i .
\end{aligned}
\end{equation} 
Since $\frac{1}{\sqrt{N}}\Phi_t^N=G_\beta*(\mu_t^N-\bar\rho_t)=G_\beta*\mu_t^N-\phi_t$, we can rewrite the third and the martingale line as
  \[
  \frac{2}{\sqrt{N}}\sum_{i=1}^N \nabla\Phi_t^N(X_t^i)\cdot\Delta_t^i\,\mathrm dt
  \qquad\text{and}\qquad
  \mathrm d\mathfrak M_t^N
  :=\frac{2}{\sqrt{N}}\sum_{i=1}^N \nabla\Phi_t^N(X_t^i)\cdot\sigma\,\mathrm dB_t^i,
\] 
so that   $\mathfrak M_t^N$  is exactly \eqref{eq:energy-martingale}.

It remains to simplify the line involving $\partial_t\phi_t$.
Using $\partial_t\phi_t=G_\beta*(\partial_t\bar\rho_t)$ and the symmetry of $G_\beta$,
\[
  \Big\langle \mu_t^N,\,\partial_t\phi_t\Big\rangle
  =\Big\langle \mu_t^N,\,G_\beta*(\partial_t\bar\rho_t)\Big\rangle
  =\Big\langle \partial_t\bar\rho_t,\,G_\beta*\mu_t^N\Big\rangle.
\]
Hence
\[
  -2\langle \mu_t^N,\partial_t\phi_t\rangle + 2\langle \partial_t\bar\rho_t,\phi_t\rangle
  =-2\big\langle \partial_t\bar\rho_t,\,G_\beta*\mu_t^N-\phi_t\big\rangle
  =-\frac{2}{\sqrt{N}}\langle \partial_t\bar\rho_t,\Phi_t^N\rangle.
\]

Since \(\bar\rho\)  solves the limit Fokker--Planck equation with drift
  \[
V_t=b(t,\cdot)+\Kop{\bar\rho_t},
\]
 for every fixed   \(t\)  and every smooth test function   \(\psi\),

\[
  \langle \partial_t\bar\rho_t,\psi\rangle
  =\langle \bar\rho_t,\mathcal L\psi\rangle+\langle \bar\rho_t,V_t\cdot\nabla\psi\rangle.
\] 
We apply this identity at time $t$ with $\psi=\Phi_t^N$ (evaluation at a fixed time, so no term involving $\partial_t\Phi_t^N$ appears) and obtain
  \[
  -2\langle \partial_t\bar\rho_t,\Phi_t^N\rangle
  =-2\langle \bar\rho_t,\mathcal L\Phi_t^N\rangle
   -2\langle \bar\rho_t,V_t\cdot\nabla\Phi_t^N\rangle.
\] 
Substituting this into \eqref{eq:dE-pre}, and using
  \[
  \Big\langle \mu_t^N,\,\mathcal L(G_\beta*\mu_t^N)- \mathcal L\phi_t\Big\rangle
  =\frac{1}{\sqrt{N}}\langle \mu_t^N,\mathcal L\Phi_t^N\rangle,
\quad
  \Big\langle \mu_t^N,\,V_t\cdot\nabla(G_\beta*\mu_t^N)-\,V_t\cdot\nabla\phi_t\Big\rangle
  =\frac{1}{\sqrt{N}}\langle \mu_t^N,V_t\cdot\nabla\Phi_t^N\rangle,
\] 
we arrive at
  
\begin{align*}
  \mathrm d\|\eta_t^N\|_{H^{-\beta}}^2
  &=
  c_{\sigma,\beta}\,\mathrm dt
  +2\sqrt{N}\langle \mu_t^N-\bar\rho_t,\mathcal L\Phi_t^N\rangle\,\mathrm dt
\\
  &\quad
  +2\sqrt{N}\langle \mu_t^N-\bar\rho_t,V_t\cdot\nabla\Phi_t^N\rangle\,\mathrm dt
  +\frac{2}{\sqrt{N}}\sum_{i=1}^N \nabla\Phi_t^N(X_t^i)\cdot\Delta_t^i\,\mathrm dt
  +\mathrm d\mathfrak M_t^N.
\end{align*} 
 
 Since $\eta_t^N=\sqrt{N}(\mu_t^N-\bar\rho_t)$, this is exactly the differential form of  the identity  \eqref{eq:energy-identity}, and integrating over $[0,t]$
yields \eqref{eq:energy-identity}.
\end{proof}

\begin{lemma}\label{lem:energy-ineq}
Let $\beta>\frac d2+2$ and let \(m\in\mathbb N\) satisfy \(m>\beta\).
Assume that
  \[
  V_T:=\sup_{t\in[0,T]}\|V_t\|_{W^{m,\infty}}<\infty.
\] 
Then there exist constants $c_0,C_0>0$ (depending only on $\beta$, $T$, $\sigma$ and $V_T$) such that for all $t\in[0,T]$,
  \begin{equation}\label{eq:energy-ineq}
\begin{aligned}
\|\eta_t^N\|_{H^{-\beta}}^2
  +c_0\int_0^t \|\eta_r^N\|_{H^{-\beta+1}}^2\,\mathrm dr
  &\le
  \|\eta_0^N\|_{H^{-\beta}}^2
  +C_0 t
  +C_0\int_0^t \|\eta_r^N\|_{H^{-\beta}}^2\,\mathrm dr \\
  &\quad
  +C_0\int_0^t \Big(\frac{1}{\sqrt{N}}\sum_{i=1}^N|\Delta_r^i|\Big)^2\,\mathrm dr
  +\mathfrak M_t^N .
\end{aligned}
\end{equation} 
\end{lemma}

\begin{proof}
Fix $t\in[0,T]$. We start from the energy identity \eqref{eq:energy-identity} and estimate the three time-integrals.

\smallskip

Let $\Sigma:=\sigma\sigma^\top$ and denote by $\lambda_\Sigma>0$ the smallest eigenvalue of $\Sigma$.
Using the Fourier characterization of $H^{-\beta}$, for any signed measure (or distribution) $\eta$ such that the expressions are finite,
\[
  2\big\langle \eta,\mathcal L(G_\beta*\eta)\big\rangle
  =-\int_{\R^d} (\xi^\top\Sigma\xi)\,(1+|\xi|^2)^{-\beta}\,|\widehat\eta(\xi)|^2\,\mathrm d\xi.
\]
Since $\xi^\top\Sigma\xi\ge \lambda_\Sigma|\xi|^2$ and
\[
  |\xi|^2(1+|\xi|^2)^{-\beta}=(1+|\xi|^2)^{-\beta+1}-(1+|\xi|^2)^{-\beta},
\]
we obtain, for all $\eta$,
\begin{equation}\label{eq:diffusion-dissipation}
  2\big\langle \eta,\mathcal L(G_\beta*\eta)\big\rangle
  \le -\lambda_\Sigma\|\eta\|_{H^{-\beta+1}}^2+\lambda_\Sigma\|\eta\|_{H^{-\beta}}^2.
\end{equation}
Applying this with $\eta=\eta_r^N$ yields
  \begin{equation}\label{eq:diffusion-bound-r}
  2\big\langle \eta_r^N,\mathcal L\Phi_r^N\big\rangle
  \le -\lambda_\Sigma\|\eta_r^N\|_{H^{-\beta+1}}^2+\lambda_\Sigma\|\eta_r^N\|_{H^{-\beta}}^2.
\end{equation} 

\smallskip

Next, write $\Phi_r^N=G_\beta*\eta_r^N$. By the multiplier estimate on $H^\beta$ (cf.\ Lemma~\ref{lem:multiplier}),
  \[
  \|V_r\cdot\nabla \Phi_r^N\|_{H^\beta}\le C_\beta \|V_r\|_{W^{m,\infty}}\|\Phi_r^N\|_{H^{\beta+1}}
  \le C_\beta V_T \|\eta_r^N\|_{H^{-\beta+1}},
\] 
where we used $\|\Phi_r^N\|_{H^{\beta+1}}=\|\eta_r^N\|_{H^{-\beta+1}}$.
By duality between $H^{-\beta}$ and $H^\beta$,
  \[
  \big|\langle \eta_r^N, V_r\cdot\nabla\Phi_r^N\rangle\big|
  \le \|\eta_r^N\|_{H^{-\beta}}\|V_r\cdot\nabla\Phi_r^N\|_{H^\beta}
  \le C_\beta V_T \|\eta_r^N\|_{H^{-\beta}}\|\eta_r^N\|_{H^{-\beta+1}}.
\] 
We apply Young's inequality with the fixed choice $\varepsilon:=\lambda_\Sigma/4$ and obtain
  \begin{equation}\label{eq:transport-bound-r}
  2\big|\langle \eta_r^N, V_r\cdot\nabla\Phi_r^N\rangle\big|
  \le \frac{\lambda_\Sigma}{4}\|\eta_r^N\|_{H^{-\beta+1}}^2
       +C_{\beta}\,\frac{V_T^2}{\lambda_\Sigma}\,\|\eta_r^N\|_{H^{-\beta}}^2 .
\end{equation} 

\smallskip

Finally, since $\beta>\frac d2+2$, Sobolev embedding gives $H^{\beta+1}\hookrightarrow W^{1,\infty}$ and hence
\[
  \|\nabla\Phi_r^N\|_\infty\le C_\beta\|\Phi_r^N\|_{H^{\beta+1}}=C_\beta\|\eta_r^N\|_{H^{-\beta+1}}.
\]
Therefore,
\[
  \Big|\frac{2}{\sqrt{N}}\sum_{i=1}^N \nabla\Phi_r^N(X_r^i)\cdot \Delta_r^i\Big|
  \le 2\|\nabla\Phi_r^N\|_\infty\Big(\frac{1}{\sqrt{N}}\sum_{i=1}^N|\Delta_r^i|\Big)
  \le 2C_\beta \|\eta_r^N\|_{H^{-\beta+1}}\Big(\frac{1}{\sqrt{N}}\sum_{i=1}^N|\Delta_r^i|\Big).
\]
Applying Young again with $\varepsilon:=\lambda_\Sigma/4$ yields
\begin{equation}\label{eq:interaction-bound-r}
  \Big|\frac{2}{\sqrt{N}}\sum_{i=1}^N \nabla\Phi_r^N(X_r^i)\cdot \Delta_r^i\Big|
  \le \frac{\lambda_\Sigma}{4}\|\eta_r^N\|_{H^{-\beta+1}}^2
+C_{\beta}\,\frac1{\lambda_\Sigma}\Big(\frac{1}{\sqrt{N}}\sum_{i=1}^N|\Delta_r^i|\Big)^2.
\end{equation}

\smallskip

We now insert \eqref{eq:diffusion-bound-r}, \eqref{eq:transport-bound-r}, and \eqref{eq:interaction-bound-r} into
\eqref{eq:energy-identity} and integrate over $r\in[0,t]$.
 The additional term \(c_{\sigma,\beta}t\) in Lemma~\ref{lem:energy-identity} is bounded by \(C_0t\), after increasing \(C_0\) by a constant depending only on \(\beta\) and \(\sigma\).
 Collecting the coercive contributions, we obtain
\begin{align*}
  \|\eta_t^N\|_{H^{-\beta}}^2
  &+\frac{\lambda_\Sigma}{2}\int_0^t \|\eta_r^N\|_{H^{-\beta+1}}^2\,\mathrm dr
\\
  &\le
  \|\eta_0^N\|_{H^{-\beta}}^2
  +C_0t
  +C_0\int_0^t \|\eta_r^N\|_{H^{-\beta}}^2\,\mathrm dr
\\
  &\quad
  +C_0\int_0^t \Big(\frac{1}{\sqrt{N}}\sum_{i=1}^N|\Delta_r^i|\Big)^2\,\mathrm dr
  +\mathfrak M_t^N,
\end{align*} 
with a constant $C_0$ depending only on $\beta$, $\sigma$ and $V_T$.
This is exactly \eqref{eq:energy-ineq} with $c_0:=\lambda_\Sigma/2$.
\end{proof}
  Lemma~\ref{lem:energy-ineq} shows that the whole problem has been reduced to bounding the three error terms on the right-hand side: the initial empirical error, the accumulated drift mismatch, and the energy martingale. The next lemma packages this reduction in a form convenient for the final estimate.
 \begin{lemma}\label{lem:gronwall-reduction}
Let $\beta>\frac d2+2$ and let \(m\in\mathbb N\) satisfy \(m>\beta\). Assume that
  \[
  V_T:=\sup_{t\in[0,T]}\|V_t\|_{W^{m,\infty}}<\infty,
\] 
and let $c_0,C_0$ be the constants in Lemma~\ref{lem:energy-ineq}. Define
  \[
  Y_T^N:=\sup_{0\le t\le T}\|\eta_t^N\|_{H^{-\beta}}^2,\qquad
  J_T^N:=\int_0^T \Big(\frac{1}{\sqrt{N}}\sum_{i=1}^N|\Delta_r^i|\Big)^2\,\mathrm dr,\qquad
  \mathfrak M_T^{N,*}:=\sup_{0\le t\le T}\big|\mathfrak M_t^N\big|.
\] 
Then there exists $C_T>0$ (depending only on $T$ and $C_0$) such that
  \begin{equation}\label{eq:reduction}
  Y_T^N
  +\int_0^T \|\eta_r^N\|_{H^{-\beta+1}}^2\,\mathrm dr
  \le
  C_T\big(1+\|\eta_0^N\|_{H^{-\beta}}^2  +J_T^N+\mathfrak M_T^{N,*}\big).
\end{equation} 
\end{lemma}
\begin{proof}
  For \(0\le t\le T\), set
\[
Y_t^N:=\sup_{0\le r\le t}\|\eta_r^N\|_{H^{-\beta}}^2,
\qquad
D_t^N:=\int_0^t\|\eta_r^N\|_{H^{-\beta+1}}^2\,\mathrm dr,
\qquad
J_t^N:=\int_0^t\Big(\frac1{\sqrt N}\sum_{i=1}^N|\Delta_r^i|\Big)^2\,\mathrm dr.
\]
Dropping the nonnegative dissipation term in  Lemma~\ref{lem:energy-ineq} , taking the supremum over \(s\le t\), and using \(J_s^N\le J_T^N\) gives
\[
Y_t^N
\le
\|\eta_0^N\|_{H^{-\beta}}^2
+C_0T
+C_0J_T^N
+\mathfrak M_T^{N,*}
+C_0\int_0^tY_r^N\,\mathrm dr.
\]
Gronwall's inequality yields
\[
Y_T^N
\le
C_T\big(1+\|\eta_0^N\|_{H^{-\beta}}^2+J_T^N+\mathfrak M_T^{N,*}\big).
\]
Returning to Lemma~\ref{lem:energy-ineq} with \(t=T\), and using the preceding bound on
\(\int_0^T\|\eta_r^N\|_{H^{-\beta}}^2\,\mathrm dr\le T Y_T^N\), gives the same bound for \(D_T^N\). Combining the two estimates proves \eqref{eq:reduction} .
\end{proof}

\subsection{Proof of Theorem~\ref{thm:empirical}}\label{subsec:proof-empirical}
 We now estimate the three terms  appearing  in \eqref{eq:reduction}.   The guiding principle is simple: the initial term is of order one after rescaling by $\sqrt N$, the interaction error is controlled by the incremental entropy bound from Theorem~\ref{thm:incremental}, and the martingale term is absorbed into the dissipation by BDG and Young's inequality.
We begin with the time-zero contribution, which is explicit because the initial data are i.i.d.
 \begin{lemma}\label{lem:init-Hminus}
Let $\beta>\frac d2$ and $(X_0^i)_{i\ge1}$ be i.i.d.\ with common law $\rho_0$.
Then
\[
  \E\big[\|\mu_0^N-\rho_0\|_{H^{-\beta}}^2\big]
  =\frac1N\Big(G_\beta(0)-\iint G_\beta(x-y)\,\rho_0(\mathrm dx)\rho_0(\mathrm dy)\Big)
  \le \frac{G_\beta(0)}{N}.
\]
In particular, $G_\beta(0)-\iint G_\beta(x-y)\,\rho_0(\mathrm dx)\rho_0(\mathrm dy)\ge0$.
\end{lemma}

\begin{proof}
By the kernel representation,
\[
  \|\mu_0^N-\rho_0\|_{H^{-\beta}}^2
  =\iint G_\beta(x-y)\,(\mu_0^N-\rho_0)(\mathrm dx)\,(\mu_0^N-\rho_0)(\mathrm dy).
\]
Expanding and taking expectations,
\[
  \E\Big[\iint G_\beta\,\mu_0^N\otimes\mu_0^N\Big]
  =\frac1{N^2}\sum_{i,j=1}^N \E\big[G_\beta(X_0^i-X_0^j)\big]
  =\frac{1}{N}G_\beta(0)+\frac{N-1}{N}\,A,
\]
where $A:=\iint G_\beta(x-y)\,\rho_0(\mathrm dx)\rho_0(\mathrm dy)$.
Moreover,
\[
  \E\Big[\int (G_\beta*\rho_0)\,\mathrm d\mu_0^N\Big]
  =\int (G_\beta*\rho_0)(x)\,\rho_0(\mathrm dx)=A.
\]
Hence $\E[\|\mu_0^N-\rho_0\|_{H^{-\beta}}^2]=\frac1N(G_\beta(0)-A)$.

To see $G_\beta(0)-A\ge0$, write in Fourier variables:
\[
  A=\int_{\R^d} (1+|\xi|^2)^{-\beta}\,|\widehat{\rho_0}(\xi)|^2\,\mathrm d\xi,
  \qquad
  G_\beta(0)=\int_{\R^d}(1+|\xi|^2)^{-\beta}\,\mathrm d\xi.
\]
Since $\rho_0$ is a probability measure, $|\widehat{\rho_0}(\xi)|\le1$, hence $A\le G_\beta(0)$.
\end{proof}
  We next estimate the interaction mismatch term $J_T^N$. This is the place where the incremental entropy control from Section~\ref{sec:incremental-proof} enters the argument quantitatively.
 \begin{lemma}\label{lem:J-term}
Assume the i.i.d.\ initial condition and let $J_T^N$ be defined in Lemma~\ref{lem:gronwall-reduction}.
Then there exists $C_T>0$ such that for all $N\ge2$,
\[
  \E[J_T^N]\le {C_T}.
\]
\end{lemma}

\begin{proof}
By definition and Fubini's theorem,
\[
  \E[J_T^N]
  =\frac1{N}\int_0^T \E\Big[\Big(\sum_{i=1}^N|\Delta_r^i|\Big)^2\Big]\,\mathrm dr.
\]
For each fixed $r$, using Minkowski's inequality in $L^2(\Omega)$ gives
\[
  \Big(\E\Big[\Big(\sum_{i=1}^N|\Delta_r^i|\Big)^2\Big]\Big)^{1/2}
  \le \sum_{i=1}^N \big(\E[|\Delta_r^i|^2]\big)^{1/2}.
\]
Integrating in time $r$ and applying Minkowski's inequality  in $L^2([0, T])$, then using Fubini's theorem again, we obtain 
\[
 \E[J_T^N]
  \le  \frac 1 N \Big(\sum_{i=1}^N \Big(\E\int_0^T|\Delta_r^i|^2\,\mathrm dr\Big)^{1/2}\Big)^2.
\]

For   \(i=1\), \(\Delta_r^1=-\Kop{\bar\rho_r}{X_r^1}\), and therefore \(|\Delta_r^1|\le\|K\|_\infty\). Hence
 \[
  \E\int_0^T|\Delta_r^1|^2\,\mathrm dr \le T\|K\|_\infty^2.
\]

 For $i\ge2$, use the incremental estimate.
Let $\Sigma:=\sigma\sigma^\top$ and write $| \sigma^{-1}z|^2:= z^\top \Sigma^{-1}z$.
Then \eqref{eq:Ri-integral} gives (with $R_i(0)=0$ under i.i.d.\ initial data)
\[
  R_i(T)=\frac12\,\E_{P^{1:i}_{[0,T]}}\int_0^T \big|\sigma^{-1}\Delta_r^i\big|^2\,\mathrm dr.
\]
By Theorem~\ref{thm:incremental}, $R_i(T)\le C_T/(i-1)$ for $i\ge2$, hence
\[
  \E\int_0^T |\Delta_r^i|^2\,\mathrm dr
  \le \|\Sigma\|_{\op}\,\E\int_0^T \big|\sigma^{-1}\Delta_r^i\big|^2\,\mathrm dr
  = 2\|\Sigma\|_{\op}\,R_i(T)
  \le \frac{C_T}{i-1},
  \qquad i\ge2.
\]
Therefore,
\[
  \sum_{i=1}^N \Big(\E\int_0^T|\Delta_r^i|^2\,\mathrm dr\Big)^{1/2}
  \le C_T\Big(1+\sum_{i=2}^N (i-1)^{-1/2}\Big)
  \le C_T\big(1+2\sqrt{N}\big),
\]
and plugging into the previous bound yields $\E[J_T^N]\le C_T$.
\end{proof}
  Finally, we control the martingale term. The key point is that the gradient of the Bessel potential is bounded by the dissipative norm appearing on the left-hand side of \eqref{eq:energy-ineq}, so the martingale contribution can be reabsorbed.
 \begin{lemma}\label{lem:K-term}
Let   \(\beta>\frac d2+2\) and let \(\mathfrak M^N\)  be defined by \eqref{eq:energy-martingale}.   For every \(\varepsilon>0\), there exists \(C_{\varepsilon,T}>0\) such that
\[
  \E\big[\mathfrak M_T^{N,*}\big]
  \le
  \varepsilon\,\E\int_0^T \|\eta_r^N\|_{H^{-\beta+1}}^2\,\mathrm dr
  +C_{\varepsilon,T},
\] 
where   \(\mathfrak M_T^{N,*}=\sup_{0\le t\le T}|\mathfrak M_t^N|\) .
\end{lemma}

\begin{proof}
  Fix \(\varepsilon>0\). Set \(\Sigma:=\sigma\sigma^\top\).
Since $\sup_{t\le T}\mathfrak M_t^N\le \sup_{t\le T}|\mathfrak M_t^N|=\mathfrak M_T^{N,*}$ , it suffices to bound   $\E[\mathfrak M_T^{N,*}]$ .
By Burkholder--Davis--Gundy,
  \[
  \E[\mathfrak M_T^{N,*}]\le C\,\E\big[\langle \mathfrak M^N\rangle_T^{1/2}\big].
\] 
From \eqref{eq:energy-martingale} and independence of $(B^i)$,
  \[
  \langle \mathfrak M^N\rangle_T
  =\frac{4}{N}\sum_{i=1}^N\int_0^T \big|\sigma^\top\nabla\Phi_r^N(X_r^i)\big|^2\,\mathrm dr
  \le {4\|\Sigma\|_{\op}}\int_0^T \|\nabla\Phi_r^N\|_\infty^2\,\mathrm dr.
\] 
By Sobolev embedding (since $\beta>\frac d2+2$),
\[
  \|\nabla\Phi_r^N\|_\infty\le C_\beta\|\Phi_r^N\|_{H^{\beta+1}}
  =C_\beta\|\eta_r^N\|_{H^{-\beta+1}},
\]
hence
  \[
  \langle \mathfrak M^N\rangle_T
  \le C_T\int_0^T \|\eta_r^N\|_{H^{-\beta+1}}^2\,\mathrm dr.
\] 
Therefore,
  \[
  \E[\mathfrak M_T^{N,*}]
  \le {C_T}\,
      \E\Big[\Big(\int_0^T \|\eta_r^N\|_{H^{-\beta+1}}^2\,\mathrm dr\Big)^{1/2}\Big].
\] 
Young's inequality yields
  \[
  \E[\mathfrak M_T^{N,*}]
  \le
  \varepsilon\,\E\int_0^T \|\eta_r^N\|_{H^{-\beta+1}}^2\,\mathrm dr
  +C_{\varepsilon,T},
\] 
after adjusting constants.
\end{proof}

\begin{proof}[Proof of Theorem~\ref{thm:empirical}]
Apply Lemma~\ref{lem:gronwall-reduction} and take expectations:
  \[
  \E\Big[\sup_{0\le t\le T}\|\eta_t^N\|_{H^{-\beta}}^2 \Big]
  +\E\int_0^T \|\eta_t^N\|_{H^{-\beta+1}}^2\,\mathrm dt
  \le C_T\Big(1+\E[\|\eta_0^N\|_{H^{-\beta}}^2]+\E[J_T^N]+\E[\mathfrak M_T^{N,*}]\Big).
\] 
Since 
Lemma~\ref{lem:init-Hminus} yields $\E[\|\eta_0^N\|_{H^{-\beta}}^2]\le C_\beta$
and  Lemma~\ref{lem:J-term} gives $\E[J_T^N]\le C_T$.
  By  Lemma~\ref{lem:K-term}, for every \(\varepsilon>0\),
\[
  \E[\mathfrak M_T^{N,*}]
  \le
  \varepsilon\,\E\int_0^T \|\eta_t^N\|_{H^{-\beta+1}}^2\,\mathrm dt
  +C_{\varepsilon,T}.
\] 
  Choose \(\varepsilon>0\) sufficiently small so that  the dissipation term  can be absorbed  into the left-hand side.   Together with Lemmas~\ref{lem:init-Hminus} and \ref{lem:J-term}, this yields   \[
  \E\Big[\sup_{0\le t\le T}\|\eta_t^N\|_{H^{-\beta}}^2\Big]
  +\E\int_0^T \|\eta_t^N\|_{H^{-\beta+1}}^2\,\mathrm dt
  \le {C_{\beta,T}}.
\] 
 Since   $\eta_t^N=\sqrt N(\mu_t^N-\bar\rho_t)$, dividing the last inequality by $N$  gives
  \[
\E\Big[\sup_{0\le t\le T}\|\mu_t^N-\bar\rho_t\|_{H^{-\beta}}^2\Big]
+\E\int_0^T \|\mu_t^N-\bar\rho_t\|_{H^{-\beta+1}}^2\,\mathrm dt
\le \frac{C_{\beta,T}}{N},
\] 
  which is exactly \eqref{eq:empirical-rate}.
\end{proof}

 \section*{Acknowledgements}
The authors would like to thank Kai Du for   helpful  discussions on the sequential interacting particle system. This work was partially supported by the National Key R\&D Program of China (Project Nos.~2024YFA1015500 and 2021YFA1002800) and by the NSFC (Grant Nos.~12595282 and 12171009).
\bibliographystyle{plain}
\bibliography{references}

\end{document}